\documentclass[a4paper,11pt]{amsart}
\usepackage{amsfonts,amssymb,amsthm,cite,amsmath,amstext,mathabx}
\usepackage{color}
\usepackage[colorlinks,linkcolor=black,citecolor=black]{hyperref}
\usepackage{comment}
\usepackage{framed}

\definecolor{shadecolor}{gray}{0.875}

\newtheorem{thrm}{Theorem}[section]
\newtheorem{lem}[thrm]{Lemma}
\newtheorem{cor}[thrm]{Corollary}
\newtheorem{prop}[thrm]{Proposition}
\newtheorem{conj}[thrm]{Conjecture}

\theoremstyle{definition}
\newtheorem{defn}[thrm]{Definition}

\newtheorem{exmple}[thrm]{Example}
\newtheorem{rmk}[thrm]{Remark}

\newenvironment{claim}
            {\par \bigskip \noindent \textbf{Claim:}}
            {$\Box$ \par \noindent}

\DeclareMathOperator{\HNF}{HNF}
\DeclareMathOperator{\Eff}{\overline{Eff}}
\DeclareMathOperator{\Nef}{Nef}
\DeclareMathOperator{\Mov}{Mov}

\DeclareMathOperator{\vol}{vol}

\DeclareMathOperator{\mc}{mc}
\DeclareMathOperator{\rmc}{rmc}
\DeclareMathOperator{\wmc}{wmc}
\DeclareMathOperator{\mob}{mob}

\DeclareMathOperator{\wmob}{wmob}

\DeclareMathOperator{\CI}{CI}

\DeclareMathOperator{\mult}{mult}

\DeclareMathOperator{\HConc}{HConc}

\DeclareMathOperator{\BC}{BC}

\title{\textbf{Positivity functions for curves on algebraic varieties}}
\author{\textsc{Brian Lehmann and Jian Xiao}}
\date{}
\begin{document}

\begin{abstract}
This is the second part of our work on Zariski decomposition structures, where we compare two different volume type functions for curve classes.  The first function is the polar transform of the volume for ample divisor classes.  The second function captures the asymptotic geometry of curves analogously to the volume function for divisors.  We prove that the two functions coincide, generalizing Zariski's classical result for surfaces to all varieties.  Our result confirms the log concavity conjecture of the first named author for weighted mobility of curve classes in an unexpected way, via Legendre-Fenchel type transforms. We also give a number of applications to birational geometry, including a refined structure theorem for the movable cone of curves.
\end{abstract}

\maketitle

\section{Introduction}

Let $X$ be a smooth complex projective variety of dimension $n$ and let $L$ be a divisor on $X$.  Perhaps the most important birational invariant of $L$ is its volume, defined as
\begin{equation*}
\vol(L) := \limsup_{m \to \infty} \frac{\dim H^{0}(X,mL)}{m^{n}/n!}.
\end{equation*}
When $X$ is a surface, the volume of $L$ can be calculated using intersection theory.  The key construction is the Zariski decomposition of \cite{zariski62}, which splits a curve $L$ into a ``positive'' part and a ``rigid'' part.  In higher dimensions as well, there is a close relationship between the asymptotic geometry of divisors and the positivity of numerical classes via volume-type functions. 

In this paper we develop an analogous theory for curve classes on arbitrary varieties: asymptotic geometry is controlled by intersection theory.  We analyze this relationship by comparing several natural volume-type functions for curve classes.  The first function involves the numerical positivity of a curve class.

\begin{defn}(see \cite[ Definition 1.1]{xiao15}) \label{defn:widehatvol}
Let $X$ be a projective variety of dimension $n$ and let $\alpha \in \Eff_{1}(X)$ be a pseudo-effective curve class.  Then the volume of $\alpha$ is defined to be
\begin{align*}
\widehat{\vol}(\alpha) = \inf_{A \textrm{ big and nef divisor class}} \left( \frac{A \cdot \alpha}{\vol(A)^{1/n}} \right)^{{n}/{n-1}}.
\end{align*}
When $\alpha$ is a curve class that is not pseudo-effective, we set $\widehat{\vol}(\alpha)=0$.
\end{defn}

This is a polar transformation of the volume function on the ample cone of divisors.  The definition is inspired by the realization that the volume of a divisor has a similar intersection-theoretic description against curves as in \cite[Theorem 2.1]{xiao15}.  It fits into a much broader picture relating positivity of divisors and curves via cone duality; see \cite{lehmannxiao2015a}.

The second function captures the asymptotic geometry of curves.  Recall that general points impose independent codimension $1$ conditions on divisors in a linear series.  Thus for a divisor $L$, one can interpret $\dim \mathbb{P}(H^{0}(X,L))$ as a measurement of how many general points are contained in sections of $L$.  Using this interpretation, we define the mobility function for curves in an analogous way.

\begin{defn}
(see \cite[Definition 1.1]{lehmann14})
Let $X$ be a projective variety of dimension $n$ and let $\alpha \in N_{1}(X)$ be a curve class with integer coefficients.  The mobility of $\alpha$ is defined to be
\begin{equation*}
\mob(\alpha):= \limsup_{m \to \infty} \frac{\max \left\{ b \in \mathbb{Z}_{\geq 0} \, \left| \,
\begin{array}{c}
\textrm{Any }b \textrm{ general points are contained}\\
\textrm{in an effective curve of class } m\alpha
\end{array} \right. \right\}}{m^{{n}/{n-1}}/n!}.
\end{equation*}
\end{defn}

There is a closely related function known as the weighted mobility which counts singular points of the curve with a ``higher weight''.  We first recall the definition of the weighted mobility count for a class $\alpha \in N_{1}(X)$ with integer coefficients (see \cite[Definition 8.7]{lehmann14}):
\begin{equation*}
\wmc(\alpha) =\sup_{\mu} \max \left\{ b \in \mathbb{Z}_{\geq 0} \, \left| \, \begin{array}{c} \textrm{there is an effective cycle of class } \mu \alpha \\
\textrm{ through any }b\textrm{ points of } X \textrm{ with}  \\ \textrm{multiplicity at least }\mu \textrm{ at each point}  \end{array} \right. \right\}.
\end{equation*}
The supremum is shown to exist in \cite{lehmann14} -- it is then clear that the supremum is achieved by some positive integer $\mu$.  We define the weighted mobility to be
\begin{equation*}
\wmob(\alpha) = \limsup_{m \to \infty} \frac{ \wmc(m\alpha) }{m^{{n}/{n-1}}}.
\end{equation*}
While the definition is slightly more complicated, the weighted mobility is easier to compute due to its close relationship with Seshadri constants.
In \cite{lehmann14}, the first named author shows that both the mobility and weighted mobility extend to continuous homogeneous functions on all of $N_{1}(X)$.

Our main theorem compares these functions.  It continues a project begun by the second named author (see \cite[Conjecture 3.1 and Theorem 3.2]{xiao15}).

\begin{thrm}
\label{thm volversusmob}
Let $X$ be a smooth projective variety of dimension $n$ and let $\alpha \in \Eff_{1}(X)$ be a pseudo-effective curve class.  Then:
\begin{enumerate}
\item $\widehat{\vol}(\alpha) = \wmob(\alpha)$.
\item $\widehat{\vol}(\alpha) \leq \mob(\alpha) \leq n! \widehat{\vol}(\alpha)$.
\item Assume Conjecture \ref{conj mobandciconj} below.  Then $\mob(\alpha) = \widehat{\vol}(\alpha)$.
\end{enumerate}
\end{thrm}

Theorem \ref{thm volversusmob} is quite surprising: it suggests that the mobility count of \emph{any} curve class is optimized by complete intersection curves; see Section \ref{section curvezar}.  
The final part of this theorem relies on the following (difficult) conjectural description of the mobility of a complete intersection class:

\begin{conj}(see \cite[Question 7.1]{lehmann14}) \label{conj mobandciconj}
Let $X$ be a smooth projective variety of dimension $n$ and let $A$ be an ample divisor on $X$.  Then
\begin{equation*}
\mob(A^{n-1}) = A^{n}.
\end{equation*}
\end{conj}

\begin{exmple}
Let $\alpha$ denote the class of a line on $\mathbb{P}^{3}$.  The mobility count of $\alpha$ is determined by the following enumerative question: what is the minimal degree of a curve through $b$ general points of $\mathbb{P}^{3}$?  The answer is unknown, even in an asymptotic sense.

Perrin \cite{perrin87} conjectured that the ``optimal'' curves (which maximize the number of points relative to their degree to the $3/2$) are complete intersections of two divisors of the same degree.  Theorem \ref{thm volversusmob} supports a vast generalization of Perrin's conjecture to all big curve classes on all smooth projective varieties.
\end{exmple}

Similar to the case of curves on algebraic surfaces, Theorem \ref{thm volversusmob} relies on the Zariski decomposition for curves on smooth projective varieties of arbitrary dimension.  
There are two such constructions -- one defined for $\widehat{\vol}$ using convex analysis as in \cite{lehmannxiao2015a}, and the other one defined for $\mob$ using the geometry of cycles as in \cite{fl14}.  The comparison between these two decompositions is the driving force behind Theorem \ref{thm volversusmob}.

\subsection{Polar transform of the volume function for divisors}
An important ingredient in the proof of Theorem \ref{thm volversusmob} is the study of the volume function for divisors from the perspective of convexity theory.  In \cite{lehmannxiao2015a} we described a general theory of convex duality for log-concave homogeneous functions on a cone $\mathcal{C}$.  The key concept is a polar transform, defining a function on the dual cone $\mathcal{C}^{*}$, which plays the role of the Legendre-Fenchel transform in classical convex analysis; we briefly recall related results in Section \ref{section preliminaries}.  The volume function fits into this abstract framework; a posteriori, this viewpoint motivates many of the well-known structure results for the volume function (such as the formula for the derivative, the Khovanskii-Teissier inequalities, the $\sigma$-decomposition, etc.).

The first step is to analyze the strict log concavity of the volume function.  We show that the volume function is strictly log concave on the big and movable cone of divisors (but on no larger cone), extending \cite[Theorem D]{bfj09}.

\begin{thrm}
\label{thm sigma strict logcon}
Let $X$ be a smooth projective variety of dimension $n$.  For any two big divisor classes $L_{1}$, $L_{2}$, we have
\begin{equation*}
\vol(L_{1} + L_{2})^{1/n} \geq \vol(L_{1})^{1/n} + \vol(L_{2})^{1/n}
\end{equation*}
with equality if and only if the (numerical) positive parts $P_{\sigma}(L_{1})$, $P_{\sigma}(L_{2})$ are proportional.  Thus the function $L \mapsto \vol(L)$ is strictly log concave on the cone of big and movable divisors.
\end{thrm}

The next step is to analyze the polar transform of the volume function.  This polar transform (denoted by $\mathfrak{M}$) is defined on the movable cone of curves by the main result of \cite{BDPP13}.  It was first defined in \cite{xiao15}.

\begin{defn}(see \cite[Definition 2.2]{xiao15})
Let $X$ be a projective variety of dimension $n$.  For any curve class $\alpha \in \Mov_{1}(X)$ define
\begin{equation*}
\mathfrak{M}(\alpha) = \inf_{L \textrm{ big divisor class}} \left( \frac{L \cdot \alpha}{\vol(L)^{1/n}} \right)^{n/n-1}.
\end{equation*}
We say that a big class $L$ computes $\mathfrak{M}(\alpha)$ if this infimum is achieved by $L$.  When $\alpha$ is a curve class that is not movable, we set $\mathfrak{M}(\alpha)=0$.
\end{defn}



The main structure theorem for $\mathfrak{M}$ relies on a refined version of a theorem of \cite{BDPP13} describing the movable cone of curves.  In \cite{BDPP13}, it is proved that the movable cone $\Mov_1(X)$ is \emph{generated} by $(n-1)$-self positive products of big divisors.  We show that $\Mov_{1}(X)$ is the \emph{closure} of the set of $(n-1)$-self positive products of big divisors on the interior of $\Mov^{1}(X)$.

\begin{thrm} \label{homeothrm}
Let $X$ be a smooth projective variety of dimension $n$ and let $\alpha$ be an interior point of $\Mov_{1}(X)$.  Then there is a unique big movable divisor class $L_\alpha$ lying in the interior of $\Mov^1 (X)$ and depending continuously on $\alpha$ such that
$\langle L_\alpha^{n-1} \rangle = \alpha$.

We then have $\mathfrak{M}(\alpha) = \vol(L_{\alpha})$.
\end{thrm}

Furthermore, we show that $\mathfrak{M}$ also admits an enumerative interpretation.  We define $\mob_{mov}$ and $\wmob_{mov}$ for curve classes analogously to $\mob$ and $\wmob$, except that we only count contributions of families whose general member is a sum of irreducible movable curves.

\begin{thrm} 
Let $X$ be a smooth projective variety of dimension $n$ and let $\alpha \in \Mov_{1}(X)^{\circ}$.  Then:
\begin{enumerate}
\item $\mathfrak{M}(\alpha) = \wmob_{mov}(\alpha)$.
\item Assume Conjecture \ref{conj mobandciconj}.  Then $\mathfrak{M}(\alpha) = \mob_{mov}(\alpha)$.
\end{enumerate}
\end{thrm}

As an interesting corollary of Theorem \ref{homeothrm}, we obtain:

\begin{cor}
Let $X$ be a projective variety of dimension $n$.  Then the rays over classes of irreducible curves which deform to dominate $X$ are dense in $\Mov_{1}(X)$.
\end{cor}

Theorem \ref{homeothrm} shows that the map $\langle -^{n-1} \rangle$ is a homeomorphism from the interior of the movable cone of divisors to the interior of the movable cone of curves.  Thus, any chamber decomposition of the movable cone of curves naturally induces a decomposition of the movable cone of divisors and vice versa.  We extend this description to the boundary of $\Mov_{1}(X)$.

\begin{thrm}
Let $X$ be a smooth projective variety and let $\alpha$ be a curve class lying on the boundary of $\Mov_{1}(X)$.  Then exactly one of the following alternatives holds:
\begin{itemize}
\item $\alpha = \langle L^{n-1} \rangle$ for a big movable divisor class $L$ on the boundary of $\Mov^{1}(X)$.
\item $\alpha \cdot M = 0$ for a movable divisor class $M$.
\end{itemize}
The homeomorphism between the interiors given by Theorem \ref{homeothrm}
extends to map the big movable divisor classes on the boundary of $\Mov^{1}(X)$ bijectively to the classes of the first type.
\end{thrm}

\subsection{Examples}

\cite{lehmannxiao2015a} defines the complete intersection cone $\CI_{1}(X)$ as the closure of the set of curve classes of the form $A^{n-1}$ for an ample divisor class $A$.  While $\CI_{1}(X)$ may not be convex, this cone plays an important role in \cite{lehmannxiao2015a} via the notion of a Zariski decomposition.

\subsubsection{Mori dream spaces}
If $X$ is a Mori dream space, then the movable cone of divisors admits a chamber structure defined via the ample cones on small $\mathbb{Q}$-factorial modifications.  This chamber structure behaves compatibly with the $\sigma$-decomposition and the volume function for divisors.

For curves we obtain a complementary picture using the movable cone of curves.  Note that $\Mov_{1}(X)$ is naturally preserved by small $\mathbb{Q}$-factorial modifications.  We then have a chamber decomposition induced by Theorem \ref{homeothrm}.  A good way to analyze the chambers is to compare the behavior of the two functions $\mathfrak{M}$ and $\widehat{\vol}$ restricted to $\Mov_{1}(X)$.
\begin{itemize}
\item By Theorem \ref{homeothrm}, a curve class in the interior of $\Mov_{1}(X)$ is the $(n-1)$-positive product of a big divisor class $L$ and $\mathfrak{M}(\alpha) = \vol(L)$.  Using the birational invariance of the volume for divisors, we see that $\mathfrak{M}$ is also invariant under small $\mathbb{Q}$-factorial modifications.
\item Using the Zariski decomposition of \cite{lehmannxiao2015a}, the movable cone of curves admits a ``chamber structure'' as a union of the complete intersection cones from small $\mathbb{Q}$-factorial modifications.  However, $\widehat{\vol}$ is not invariant under small $\mathbb{Q}$-factorial modifications but changes to reflect the differing structure of the pseudo-effective cone of curves.
\end{itemize}
We show that $\widehat{\vol}$ reaches its minimum value $\widehat{\vol}(\alpha) = \mathfrak{M}(\alpha)$ precisely on the complete intersection cone of $X$, and then increases on the chambers corresponding to birational models of $X$.  In this way $\widehat{\vol}$ and $\mathfrak{M}$ are the right tools for understanding the birational geometry of curves on Mori dream spaces.

\subsubsection{Toric varieties}

Suppose that $X$ is a simplicial projective toric variety of dimension $n$ defined by a fan $\Sigma$.  A class $\alpha$ in the interior of the movable cone of curves corresponds to a positive Minkowski weight on the rays of $\Sigma$.  A fundamental theorem of Minkowski attaches to such a weight a polytope $P_{\alpha}$ whose facet normals are the rays of $\Sigma$ and whose facet volumes are determined by the weights.  In fact, Minkowski's construction exactly corresponds to the bijection of Theorem \ref{homeothrm}.

\begin{lem}
If $L$ denotes the big and movable divisor class corresponding to the polytope $P_{\alpha}$ then $\langle L^{n-1} \rangle = \alpha$.  Thus $\mathfrak{M}(\alpha) = n! \vol(P_{\alpha})$.
\end{lem}

When $\alpha$ happens to be in the complete intersection cone, this quantity also agrees with $\widehat{\vol}(\alpha)$.


\subsection{Outline of paper} \label{outlinesec}

In this paper we will work with projective varieties over $\mathbb{C}$, but related results can be also adjusted to arbitrary algebraically closed fields and compact K\"ahler manifolds. We give a general framework for this extension in Section \ref{section preliminaries}.
In Section \ref{section preliminaries} we briefly recall the general duality framework in \cite{lehmannxiao2015a}, and explain how the proofs
can be adjusted to arbitrary algebraically closed fields and compact hyperk\"ahler manifolds.
In Section \ref{section movcone}, we give a refined structure of the movable cone of curves and generalize several results on big and nef divisors to big and movable divisors. Section \ref{toric section} discusses toric varieties, showing some relationships with convex geometry. Section \ref{ci and mov section} compares the complete intersection and movable cone of curves.  In Section \ref{section mobility} we compare the mobility function and $\widehat{\vol}$, finishing the proof of the main result.

\subsection*{Acknowledgements}
We thank M.~Jonsson for his many helpful comments.  Some of the material on toric varieties was worked out in a conversation with J.~Huh, and we are very grateful for his help.  Lehmann would like to thank C.~Araujo, M.~Fulger, D.~Greb, R.~Lazarsfeld, S.~Payne, D.~Treumann, and D.~Yang for helpful conversations.  Xiao would like to thank his supervisor J.-P.~Demailly for suggesting an intersection-theoretic approach to study volume function, and thank W.~Ou for helpful conversations, and thank the China Scholarship Council for the support.

\section{Preliminaries}
\label{section preliminaries}

\subsection{Positivity cones}
In this section, we first fix some notations over a projective variety $X$:
\begin{itemize}
\item $N^1(X)$: the real vector space of numerical classes of divisors;
\item $N_1(X)$: the real vector space of numerical classes of curves;
\item $\Eff^1(X)$: the cone of pseudo-effective divisor classes;
\item $\Nef^1(X)$: the cone of nef divisor classes;
\item $\Mov^1(X)$: the cone of movable divisor classes;
\item $\Eff_1(X)$: the cone of pseudo-effective curve classes;
\item $\Mov_1(X)$: the cone of movable curve classes, equivalently by \cite{BDPP13} the dual of $\Eff^{1}(X)$;
\item $\CI_1(X)$: the closure of the set of all curve classes of the form $A^{n-1} $ for an ample divisor $A$.
\end{itemize}

With only a few exceptions, capital letters $A,B,D,L$ will denote $\mathbb{R}$-Cartier divisor classes and greek letters $\alpha,\beta,\gamma$ will denote curve classes.  For two curve classes $\alpha, \beta$, we write $\alpha\succeq \beta$ (resp. $\alpha\preceq \beta$) to denote that $\alpha-\beta$ (resp. $\beta-\alpha$) belongs to $\Eff_{1}(X)$.  We will do similarly for divisor classes, or two elements of a cone $\mathcal{C}$ if the cone is understood.

We will use the notation $\langle - \rangle$ for the positive product as in \cite{BDPP13}, \cite{bfj09} and \cite{Bou02}. We make a few remarks on this construction for singular projective varieties.  Suppose that $X$ has dimension $n$.  Then $N_{n-1}(X)$ denotes the vector space of $\mathbb{R}$-classes of Weil divisors up to numerical equivalence as in \cite[Chapter 19]{fulton84}.  In this setting, the $1$st and $(n-1)$st positive product should be interpreted respectively as maps $\Eff^{1}(X) \to N_{n-1}(X)$ and $\Eff^{1}(X)^{\times n-1} \to \Mov_{1}(X)$.  We will also let $P_{\sigma}(L)$ denote the positive part in this sense -- that is, pullback $L$ to better and better Fujita approximations, take its positive part, and push the numerical class forward to $X$ as a numerical Weil divisor class.  With these conventions, we still have the crucial result of \cite{bfj09} and \cite{lm09} that the derivative of the volume is controlled by intersecting against the positive part.

We define the movable cone of divisors $\Mov^{1}(X)$ to be the subset of $\Eff^{1}(X)$ consisting of divisor classes $L$ such that $N_{\sigma}(L) = 0$ and $P_{\sigma}(L) = L \cap [X]$.    On any projective variety, by \cite[Example 19.3.3]{fulton84} capping with $X$ defines an injective linear map $N^{1}(X) \to N_{n- 1}(X)$.
Thus if $D,L \in \Mov^{1}(X)$ have the same positive part in $N_{n-1}(X)$, then by the injectivity of the capping map we must have $D=L$.

To extend our results (especially the results in Section \ref{section movcone}) to arbitrary compact K\"ahler manifolds, we need to deal with transcendental objects which are not given by divisors or curves. Let $X$ be a compact K\"ahler manifold of dimension $n$.  By analogue with the projective situation, we need to deal with the following spaces and positive cones:
\begin{itemize}
\item $H^{1,1}_{\BC}(X, \mathbb{R})$: the real Bott-Chern cohomology group of bidegree $(1,1)$;
\item $H^{n-1,n-1}_{\BC}(X, \mathbb{R})$: the real Bott-Chern cohomology group of bidegree $(n-1,n-1)$;
\item $\mathcal{N}(X)$: the cone of pseudo-effective $(n-1,n-1)$-classes;
\item $\mathcal{M}(X)$: the cone of movable $(n-1,n-1)$-classes;
\item $\overline{\mathcal{K}}(X)$: the cone of nef $(1,1)$-classes, equivalently the closure of the K\"ahler cone;
\item $\mathcal{E}(X)$:  the cone of pseudo-effective $(1,1)$-classes.
\end{itemize}
Recall that we call a Bott-Chern class pseudo-effective if it contains a $d$-closed positive current, and call an $(n-1,n-1)$-class movable if it is contained in the closure of the cone generated by the classes of the form $\mu_{*}(\widetilde{\omega}_1 \wedge...\wedge \widetilde{\omega}_{n-1})$ where $\mu: \widetilde{X}\rightarrow X$ is a modification and $\widetilde{\omega}_1,...,\widetilde{\omega}_{n-1}$ are K\"ahler metrics on $\widetilde{X}$. For the basic theory of positive currents, we refer the reader to \cite{Dem}.

\subsection{Fields of characteristic $p$}
\label{charp background sec}
Almost all the results in the paper will hold for smooth varieties over an arbitrary algebraically closed field.  The necessary technical generalizations are verified in the following references:
\begin{itemize}
\item The existence of Fujita approximations over an arbitrary algebraically closed field is proved in \cite{takagi07}.
\item The basic properties of the $\sigma$-decomposition in positive characteristic are considered in \cite{mustata11}.
\item The results of \cite{Cut13} lay the foundations of the theory of positive products and volumes over an arbitrary field.
\item \cite{fl14} describes how to extend \cite{BDPP13} and most of the results of \cite{bfj09} over an arbitrary algebraically closed field.  In particular the description of the derivative of the volume function in \cite[Theorem A]{bfj09} holds for smooth varieties in any characteristic.
\end{itemize}

\subsection{Compact K\"ahler manifolds}
\label{kahler background sec}
The following results enable us to extend our results in Section \ref{section movcone} and Section \ref{ci and mov section} to arbitrary compact hyperk\"ahler manifolds.
\begin{itemize}
\item The theory of positive intersection products for pseudo-effective $(1,1)$-classes has been developed by \cite{Bou02, BDPP13, BEGZ10MAbig}.
\item Divisorial Zariski decomposition for pseudo-effective $(1,1)$-classes has been studied in \cite{Bou04, BDPP13}.
\item By \cite[Theorem 10.12]{BDPP13}, the transcendental analogues of the results in \cite{bfj09, BDPP13} are true for compact hyperk\"ahler manifolds. In particular, we have the cone duality $\mathcal{E}^* = \mathcal{M}$ and the description of the derivative of the volume for pseudo-effective $(1,1)$-classes.
\end{itemize}

\subsection{Polar transforms}  \label{legendresection}
As explained in the introduction, our results use convex analysis, and in particular a Legendre-Fenchel type transform for functions defined on a cone.  We briefly recall some definitions and results from \cite{lehmannxiao2015a} which will be used to study the function $\mathfrak{M}$.

\subsubsection{Duality transforms}
Let $V$ be a finite-dimensional $\mathbb{R}$-vector space of dimension $n$,  and let $V^{*}$ be its dual. We denote the pairing of $w^* \in V^*$ and $v \in V$ by $w^* \cdot v$.  Let $\mathcal{C}\subset V$ be a proper closed convex cone of full dimension and let $\mathcal{C}^* \subset V^*$ denote the dual cone of $\mathcal{C}$.
We let $\HConc_s (\mathcal{C})$ denote the collection of functions $f: \mathcal{C} \to \mathbb{R}$ satisfying:
\begin{itemize}
  \item $f$ is upper-semicontinuous and homogeneous of weight $s>1$;
  \item $f$ is strictly positive in the interior of $\mathcal{C}$ (and hence non-negative on $\mathcal{C}$);
  \item $f$ is $s$-concave: for any $v,x \in \mathcal{C}$ we have $f(v)^{1/s} + f(x)^{1/s} \leq f(v+x)^{1/s}$.
\end{itemize}

The polar transform $\mathcal{H}$ associates to a function $f \in \HConc_{s}(\mathcal{C})$ the function $\mathcal{H}f: \mathcal{C}^{*} \to \mathbb{R}$ defined as
\begin{align*}
\mathcal{H} f (w^*):= \inf_{v\in \mathcal{C}^{\circ}} \left(\frac{w^* \cdot v}{f(v)^{1/s}}\right)^{s/s-1}.
\end{align*}
The definition is unchanged if we instead vary $v$ over all elements of $\mathcal{C}$ where $f$ is positive.  It is not hard to see that $\mathcal{H}^{2}f = f$ for any $f \in \HConc_{s}(\mathcal{C})$.



It will be crucial to understand which points obtain the infimum in the definition of $\mathcal{H}f$.

\begin{defn}
Let $f \in \HConc_{s}(\mathcal{C})$.  For any $w^* \in \mathcal{C}^{*}$, we define $G_{w^*}$ to be the set of all $v \in \mathcal{C}$ which satisfy $f(v)>0$ and which achieve the infimum in the definition of $\mathcal{H}f(w^*)$, so that
\begin{align*}
\mathcal{H} f (w^*) = \left( \frac{w^* \cdot v}{f(v)^{1/s}} \right)^{s/s-1}.
\end{align*}
\end{defn}

\begin{rmk}
The set $G_{w^*}$ is the analogue of supergradients of concave functions. In particular, we know the differential of $\mathcal{H}f$ at $w^*$ lies in $G_{w^*}$ if $\mathcal{H}f$ is differentiable.
\end{rmk}

We next identify the collection of points where $f$ is controlled by $\mathcal{H}$.

\begin{defn}
Let $f \in \HConc_{s}(\mathcal{C})$.  We define $\mathcal{C}_{f}$ to be the set of all $v \in \mathcal{C}$ such that $v \in G_{w^*}$ for some $w^{*} \in \mathcal{C}$ satisfying $\mathcal{H}f(w^*)>0$.
\end{defn}




We say that $f \in \HConc_{s}(\mathcal{C})$ is differentiable if it is $\mathcal{C}^{1}$ on $\mathcal{C}^{\circ}$.  In this case we define the function
\begin{align*}
D: \mathcal{C}^{\circ}  \to V^{*} \qquad \qquad \textrm{by} \qquad \qquad
v  \mapsto \frac{Df(v)}{s}.
\end{align*}
We will need to understand the behaviour of the derivative along the boundary.






\begin{defn}
We say that $f \in \HConc_{s}(\mathcal{C})$ is $+$-differentiable if $f$ is $\mathcal{C}^{1}$ on $\mathcal{C}^{\circ}$
and the derivative on $\mathcal{C}^{\circ}$ extends to a continuous function on all of $\mathcal{C}_{f}$.
\end{defn}


\begin{rmk} \label{extensiontoboundaryrmk}
For $+$-differentiable functions $f$, we define the function $D: \mathcal{C}_{f} \to V^{*}$ by extending continuously from $\mathcal{C}^{\circ}$.
\end{rmk}

\subsection{Teissier proportionality and strict log concavity}
In \cite{lehmannxiao2015a}, we gave some conditions which are equivalent to the strict log concavity.

\begin{defn}
Let $f \in \HConc_{s}(\mathcal{C})$ be $+$-differentiable and let $\mathcal{C}_{T}$ be a non-empty subcone of $\mathcal{C}_{f}$.  We say that $f$ satisfies Teissier proportionality with respect to $\mathcal{C}_{T}$ if for any $v,x \in \mathcal{C}_{T}$ satisfying
\begin{equation*}
D(v) \cdot x = f(v)^{s-1/s} f(x)^{1/s}
\end{equation*}
we have that $v$ and $x$ are proportional.
\end{defn}

Note that we do not assume that $\mathcal{C}_{T}$ is convex -- indeed, in examples it is important to avoid this condition.  However, since $f$ is defined on the convex hull of $\mathcal{C}_{T}$, we can (somewhat abusively) discuss the strict log concavity of $f|_{\mathcal{C}_{T}}$:

\begin{defn} \label{defn generalized log concave}
Let $\mathcal{C}' \subset \mathcal{C}$ be a (possibly non-convex) subcone.  We say that $f$ is strictly log concave on $\mathcal{C}'$ if
\begin{align*}
f(v)^{1/s} + f(x)^{1/s}< f(v+x)^{1/s}
\end{align*}
holds whenever $v, x\in \mathcal{C}'$ are not proportional.  Note that this definition makes sense even when $\mathcal{C}'$ is not itself convex.
\end{defn}

\begin{thrm}
\label{thrm teissier}
Let $f \in \HConc_{s}(\mathcal{C})$ be $+$-differentiable.  For any non-empty subcone $\mathcal{C}_{T}$ of $\mathcal{C}_{f}$, consider the following conditions:
\begin{enumerate}
\item The restriction $f|_{\mathcal{C}_{T}}$ is strictly log concave (in the sense defined above).
\item $f$ satisfies Teissier proportionality with respect to $\mathcal{C}_{T}$.
\item The restriction of $D$ to $\mathcal{C}_{T}$ is injective.
\end{enumerate}
Then we have (1) $\implies$ (2) $\implies$ (3).  If $\mathcal{C}_{T}$ is convex, then we have (2) $\implies$ (1).  If $\mathcal{C}_{T}$ is an open subcone, then we have (3) $\implies$ (1).
\end{thrm}

\subsection{Sublinear boundary conditions}

Under certain conditions we can control the behaviour of $\mathcal{H}f$ near the boundary, and thus obtain the continuity.

\begin{defn}
Let $f \in \HConc_{s}(\mathcal{C})$ and let $\alpha \in (0,1)$.  We say that $f$ satisfies the sublinear boundary condition of order $\alpha$ if for any non-zero $v$ on the boundary of $\mathcal{C}$ and for any $x$ in the interior of $\mathcal{C}$, there exists a constant $C:=C(v, x)>0$ such that $f(v+\epsilon x)^{1/s}\geq C\epsilon^\alpha$.
\end{defn}

Note that the condition is always satisfied at $v$ if $f(v)>0$.  Furthermore, the condition is satisfied for any $v,x$ with $\alpha = 1$ by homogeneity and log-concavity, so the crucial question is whether we can decrease $\alpha$ slightly.

Using this sublinear condition, we get the vanishing of $\mathcal{H}f$ along the boundary.

\begin{prop} \label{sublinearcontinuity}
Let $f \in \HConc_{s}(\mathcal{C})$ satisfy the sublinear boundary condition of order $\alpha$. Then $\mathcal{H}f$ vanishes along the boundary. As a consequence, $\mathcal{H}f$ extends to a continuous function over $V^*$ by setting
$\mathcal{H}f=0$ outside $\mathcal{C}^*$.
\end{prop}

\begin{rmk}
If $f$ satisfies the sublinear condition, then $\mathcal{C}_{\mathcal{H}f} ^* =\mathcal{C}^{*\circ}$.
\end{rmk}

\subsection{Formal Zariski decompositions}

The Legendre-Fenchel transform relates the strict concavity of a function to the differentiability of its transform.  The transform $\mathcal{H}$ will play the same role in our situation; however, one needs to interpret the strict concavity slightly differently.  We will encapsulate this property using the notion of a Zariski decomposition.

\begin{defn}\label{definition formal zariski}
Let $f \in \HConc_{s}(\mathcal{C})$ and let $U \subset \mathcal{C}$ be a non-empty subcone.  We say that $f$ admits a strong Zariski decomposition with respect to $U$ if:
\begin{enumerate}
\item For every $v \in \mathcal{C}_{f}$ there are unique elements $p_{v} \in U$ and $n_{v} \in \mathcal{C}$ satisfying
\begin{equation*}
v = p_{v} + n_{v} \qquad \qquad \textrm{and} \qquad \qquad f(v) = f(p_{v}).
\end{equation*}
We call the expression $v = p_v + n_v$ the Zariski decomposition of $v$, and call $p_v$ the positive part and $n_v$ the negative part of $v$.
\item For any $v,w \in \mathcal{C}_f$ satisfying $v+w \in \mathcal{C}_f$ we have
\begin{equation*}
f(v)^{1/s} + f(w)^{1/s} \leq f(v+w)^{1/s}
\end{equation*}
with equality only if $p_{v}$ and $p_{w}$ are proportional.
\end{enumerate}
\end{defn}

In \cite{lehmannxiao2015a}, we proved the following theorem linking the existence of Zariski decomposition structure with differentiability.

\begin{thrm} \label{strong zariski equivalence}
Let $f \in \HConc_{s}(\mathcal{C})$.
Then we have the following results:
\begin{itemize}
\item If $f$ is $+$-differentiable, then $\mathcal{H}f$ admits a strong Zariski decomposition  with respect to the cone $D(\mathcal{C}_{f}) \cup \{0\}$.
\item If $\mathcal{H}f$ admits a strong Zariski decomposition with respect to a cone $U$, then $f$ is differentiable.
\end{itemize}
In the first situation, one can construct the positive part of $w^{*}$ by choosing any $v \in G_{w^{*}}$ with $f(v) > 0$ and choosing $p_{w^{*}}$ to be the unique element of the ray spanned by $D(v)$ with $\mathcal{H}f(p_{w^{*}}) = \mathcal{H}f(w^{*})$.
\end{thrm}

Under some additional conditions, we can get the continuity of the Zariski decompositions.

\begin{thrm}\label{thrm positive part continuity}
Let $f \in \HConc_{s}(\mathcal{C})$ be $+$-differentiable.  Then the function taking an element $w^{*} \in \mathcal{C}^{* \circ}$ to its positive part $p_{w^*}$ is continuous.

If furthermore $G_{v} \cup \{ 0 \}$ is a unique ray for every $v \in \mathcal{C}_{f}$ and $\mathcal{H}f$ is continuous on all of $\mathcal{C}^{*}_{\mathcal{H}f}$, then the Zariski decomposition is continuous on all of $\mathcal{C}^{*}_{\mathcal{H}f}$.
\end{thrm}

\subsection{Zariski decomposition for curves}
\label{section curvezar}
In \cite{lehmannxiao2015a}, as an application of the above formal Zariski decomposition to the situation:
\begin{align*}
\mathcal{C}=\Nef^1 (X),\quad f=\vol, \quad \mathcal{C}^*=\Eff_1 (X), \quad \mathcal{H}f=\widehat{\vol},
\end{align*}
we obtain the Zariski decomposition for curves. The following result is important in the proof of Theorem \ref{thm volversusmob}.

\begin{defn}
\label{def zariski decomposition}
Let $X$ be a projective variety of dimension $n$ and let $\alpha \in \Eff_{1}(X)^{\circ}$ be a big curve class. Then a Zariski decomposition for $\alpha$ is a decomposition
\begin{align*}
\alpha = B^{n-1} + \gamma
\end{align*}
where $B$ is a big and nef $\mathbb{R}$-Cartier divisor class, $\gamma$ is pseudo-effective, and $B \cdot \gamma = 0$.  We call $B^{n-1}$ the ``positive part'' and $\gamma$ the ``negative part" of the decomposition.
\end{defn}

\begin{thrm}
\label{thm curve decomposition}
Let $X$ be a projective variety of dimension $n$ and let $\alpha \in \Eff_{1}(X)^{\circ}$ be a big curve class. Then $\alpha$ admits a unique Zariski decomposition $\alpha = B_{\alpha} ^{n-1} + \gamma$.
Furthermore,
\begin{equation*}
\widehat{\vol}(\alpha) = \widehat{\vol}(B_{\alpha}^{n-1}) = \vol(B_{\alpha})
\end{equation*}
and $B_{\alpha}$ is the unique big and nef divisor class with this property satisfying $B_{\alpha}^{n-1} \preceq \alpha$.  The class $B_{\alpha}$ depends continuously on $\alpha$. 
\end{thrm}

\begin{rmk}
As explained in \cite[Remark 5.1]{lehmannxiao2015a}, the above result holds in the K\"ahler setting -- we have a similar decomposition for any interior point of the pseudo-effective $(n-1,n-1)$-cone $\mathcal{N}$.
\end{rmk}

\section{Positive products and movable curves}
\label{section movcone}

In this section, we study the movable cone of curves and its relationship to the positive product of divisors.  A key tool in this study is the following function of \cite[Definition 2.2]{xiao15}:

\begin{defn}[see \cite{xiao15} Definition 2.2]
Let $X$ be a projective variety of dimension $n$.  For any curve class $\alpha \in \Mov_{1}(X)$ define
\begin{equation*}
\mathfrak{M}(\alpha) = \inf_{L \textrm{ big divisor class}} \left( \frac{L \cdot \alpha}{\vol(L)^{1/n}} \right)^{n/n-1}.
\end{equation*}
We say that a big class $L$ computes $\mathfrak{M}(\alpha)$ if this infimum is achieved by $L$.  When $\alpha$ is a curve class that is not movable, we set $\mathfrak{M}(\alpha)=0$.
\end{defn}

In other words, $\mathfrak{M}$ is the function on $\Mov_{1}(X)$ defined as the polar transform of the volume function on $\Eff^{1}(X)$.  Dually, we can think of the volume function on divisors as the polar transform of $\mathfrak{M}$; this viewpoint allows us to apply the general theory of convexity developed in \cite{lehmannxiao2015a} to $\vol$. 

In this section we first prove some new results concerning the volume function for divisors.  We will then return to the study of $\mathfrak{M}$ below, where we show that it measures the volume of the ``$(n-1)$st root'' of $\alpha$.

\subsection{The volume function on big and movable divisors}
We first extend several well known results on big and nef divisors to big and movable divisors.
The key will be an extension of Teissier proportionality theorem for big and nef divisors (see \cite{lehmannxiao2015a, bfj09}) to big and movable divisors.

\begin{lem} \label{lemma psefconstantbound}
Let $X$ be a projective variety of dimension $n$.  Let $L_{1}$ and $L_{2}$ be big movable divisor classes.  Set $s$ to be the largest real number such that $L_{1} - sL_{2}$ is pseudo-effective.  Then
\begin{equation*}
s^{n} \leq \frac{\vol(L_{1})}{\vol(L_{2})}
\end{equation*}
with equality if and only if $L_{1}$ and $L_{2}$ are proportional.
\end{lem}

\begin{proof}
We first prove the case when $X$ is smooth.  Certainly we have $\vol(L_{1}) \geq \vol(sL_{2}) = s^{n}\vol(L_{2})$.  If they are equal, then since $sL_{2}$ is movable and $L_{1}-sL_{2}$ is pseudo-effective we get a Zariski decomposition of $$L_1 = sL_2 +(L_1 -sL_2)$$
in the sense of \cite{fl14}. By \cite[Proposition 5.3]{fl14}, this decomposition coincides with the numerical version of the $\sigma$-decomposition of \cite{Nak04} so that $P_{\sigma}(L_{1}) = sL_{2}$.  Since $L_{1}$ is movable, we obtain equality $L_{1}=sL_{2}$.

For arbitrary $X$, let $\phi: X' \to X$ be a resolution.  The inequality follows by pulling back $L_{1}$ and $L_{2}$ and replacing them by their positive parts.  Indeed using the numerical analogue of \cite[III.1.14 Proposition]{Nak04} we see that $\phi^{*}L_{1} - sP_{\sigma}(\phi^{*}L_{2})$ is pseudo-effective if and only if $P_{\sigma}(\phi^{*}L_{1}) - sP_{\sigma}(\phi^{*}L_{2})$ is pseudo-effective, so that $s$ can only go up under this operation.   To characterize the equality, recall that if $L_{1}$ and $L_{2}$ are movable and $P_{\sigma}(\phi^{*}L_{1}) = sP_{\sigma}(\phi^{*}L_{2})$ as elements of $N_{n-1}(X)$, then $L_{1} = sL_{2}$ as elements of $N^{1}(X)$ by the injectivity of the capping map.
\end{proof}


\begin{prop} \label{prop teissier movable}
Let $X$ be a projective variety of dimension $n$.  Let $L_{1},L_{2}$ be big and movable divisor classes.  Then
\begin{equation*}
\langle L_{1}^{n-1} \rangle \cdot L_{2} \geq \vol(L_{1})^{n-1/n} \vol(L_{2})^{1/n}
\end{equation*}
with equality if and only if $L_{1}$ and $L_{2}$ are proportional.
\end{prop}

\begin{proof}
We first suppose $X$ is smooth.  Set $s_{L}$ to be the largest real number such that $L_{1} - s_{L}L_{2}$ is pseudo-effective, and fix an ample divisor $H$ on $X$.

For any $\epsilon > 0$, by taking sufficiently good Fujita approximations we may find a birational map $\phi_{\epsilon}: Y_{\epsilon} \to X$ and ample divisor classes $A_{1,\epsilon}$ and $A_{2,\epsilon}$ such that
\begin{itemize}
\item $\phi_{\epsilon}^{*}L_{i} - A_{i,\epsilon}$ is pseudo-effective for $i=1,2$;
\item $\vol(A_{i,\epsilon}) > \vol(L_i) - \epsilon$ for $i=1,2$;
\item $\phi_{\epsilon *} A_{i,\epsilon}$ is in an $\epsilon$-ball around $L_{i}$ for $i=1,2$.
\end{itemize}
Furthermore:
\begin{itemize}
\item By applying the argument of \cite[Theorem 6.22]{fl14}, we may ensure
\begin{equation*}
\phi_{\epsilon}^{*}(\langle L_{1}^{n-1} \rangle - \epsilon H^{n-1}) \preceq A_{1,\epsilon}^{n-1} \preceq \phi_{\epsilon}^{*}(\langle L_{1}^{n-1} \rangle + \epsilon H^{n-1}).
\end{equation*}
\item Set $s_{\epsilon}$ to be the largest real number such that $A_{1,\epsilon} - s_{\epsilon}A_{2,\epsilon}$ is pseudo-effective.  Then we may ensure that $s_{\epsilon} < s_{L} + \epsilon$.
\end{itemize}
By the Khovanskii-Teissier inequality for nef divisor classes, we have
$$
 (A_{1,\epsilon}^{n-1} \cdot A_{2,\epsilon})^{n/n-1}\geq \vol(A_{1,\epsilon}) \vol(A_{2,\epsilon})^{1/n-1}.
$$
Note that $\langle L^{n-1} \rangle \cdot L_{2}$ is approximated by $A_{1,\epsilon}^{n-1} \cdot A_{2,\epsilon}$ by the projection formula.  Taking a limit as $\epsilon$ goes to $0$, we see that
\begin{equation*} \label{eqstar}
\langle L_{1}^{n-1} \rangle \cdot L_{2} \geq \vol(L_{1})^{n-1/n} \vol(L_{2})^{1/n}.  \tag{$\star$}
\end{equation*}
On the other hand, the Diskant inequality for big and nef divisors in \cite[Theorem F]{bfj09} implies that
\begin{align*}
(A_{1,\epsilon}^{n-1} \cdot A_{2,\epsilon})^{n/n-1} - & \vol(A_{1,\epsilon}) \vol(A_{2,\epsilon})^{1/n-1} \\
& \geq \left(  (A_{1,\epsilon}^{n-1} \cdot A_{2,\epsilon})^{1/n-1} - s_{\epsilon} \vol(A_{2,\epsilon})^{1/n-1}  \right)^{n} \\
& \geq \left( (A_{1,\epsilon}^{n-1} \cdot A_{2,\epsilon})^{1/n-1} - (s_{L} + \epsilon) \vol(A_{2,\epsilon})^{1/n-1}  \right)^{n}.
\end{align*}
Taking a limit as $\epsilon$ goes to $0$ again, we see that
\begin{align*}
(\langle L_{1}^{n-1} \rangle \cdot L_{2} )^{n/n-1} - & \vol(L_{1}) \vol(L_{2})^{1/n-1} \\
& \geq \left(  (\langle L_1 ^{n-1} \rangle \cdot L_{2} )^{1/n-1} - s_{L} \vol(L_{2})^{1/n-1}  \right)^{n}.
\end{align*}
Thus we extend the Diskant inequality to big and movable divisor classes.
Lemma \ref{lemma psefconstantbound}, equation (\ref{eqstar}) and the above Diskant inequality together show that
\begin{equation*}
\langle L_{1}^{n-1} \rangle \cdot L_{2} = \vol(L_{1})^{n-1/n} \vol(L_{2})^{1/n}
\end{equation*}
if and only if $L_{1}$ and $L_{2}$ are proportional.

Now suppose $X$ is singular.  The inequality can be computed by passing to a resolution $\phi: X' \to X$ and replacing $L_{1}$ and $L_{2}$ by their positive parts, since the left hand side can only decrease under this operation.  To characterize the equality, recall that if $L_{1}$ and $L_{2}$ are movable and $P_{\sigma}(\phi^{*}L_{1}) = sP_{\sigma}(\phi^{*}L_{2})$ as elements of $N_{n-1}(X)$, then $L_{1} = sL_{2}$ as elements of $N^{1}(X)$ by the injectivity of the capping map.
\end{proof}

\begin{rmk}
As a byproduct of the proof above, we get the Diskant inequality for big and movable divisor classes.
\end{rmk}

\begin{rmk}\label{rmk teissier prop movable}
In the analytic setting, applying Proposition \ref{prop teissier movable} and the same method as \cite{lehmannxiao2015a},
it is not hard to generalize Proposition \ref{prop teissier movable} to any number of big and movable divisor classes provided we have sufficient regularity for degenerate Monge-Amp\`ere equations in big classes:
\begin{itemize}
\item Let $L_1, ..., L_n$ be $n$ big divisor classes over a smooth complex projective variety $X$, then we have
\begin{align*}
\langle L_1 \cdot ... \cdot L_n \rangle \geq \vol(L_1)^{1/n}\cdot ... \cdot \vol(L_n)^{1/n}
\end{align*}
where the equality is obtained if and only if $P_{\sigma}(L_1),...,P_{\sigma}(L_n)$ are proportional.
\end{itemize}
We only need to characterize the equality situation. To see this, we need the fact that the above positive intersection $\langle L_1 \cdot ... \cdot L_n \rangle$ depends only on the positive parts $P_{\sigma}(L_i)$, which follows from the analytic construction of positive product \cite[Proposition 3.2.10]{Bou02}. Then by the method in \cite{lehmannxiao2015a}
where we apply \cite{BEGZ10MAbig} or \cite[Theorem D]{demailly2014holder},
we reduce it to the case of a pair of divisor classes, e.g. we get
\begin{align*}
\langle P_{\sigma}(L_1) ^{n-1} \cdot P_{\sigma}(L_2)\rangle = \vol(L_1)^{n-1/n}\vol(L_2)^{1/n}.
\end{align*}
By the definition of positive product we always have
\begin{align*}
\langle P_{\sigma}(L_1) ^{n-1} \cdot P_{\sigma}(L_2)\rangle \geq \langle P_{\sigma}(L_1) ^{n-1} \rangle \cdot P_{\sigma}(L_2)
\geq \vol(L_1)^{n-1/n}\vol(L_2)^{1/n},
\end{align*}
this then implies the equality
\begin{align*}
\langle P_{\sigma}(L_1) ^{n-1}\rangle  \cdot P_{\sigma}(L_2) = \vol(L_1)^{n-1/n}\vol(L_2)^{1/n}.
\end{align*}
By Proposition \ref{prop teissier movable}, we immediately obtain the desired result.
\end{rmk}

\begin{cor}
\label{cor mov-intersection after}
Let $X$ be a smooth projective variety of dimension $n$. Let $\alpha \in \Mov_{1}(X)$ be a big movable curve class.
All big divisor classes $L$ satisfying $\alpha=\langle L^{n-1}\rangle$ have the same positive part $P_\sigma(L)$.
\end{cor}

\begin{proof}
Suppose $L_{1}$ and $L_{2}$ have the same positive product.  We have $\vol(L_{1}) = \langle L_{2}^{n-1} \rangle \cdot L_{1}$ so that $\vol(L_{1}) \geq \vol(L_{2})$.  By symmetry we obtain the reverse inequality, hence equality everywhere, and we conclude by Proposition \ref{prop teissier movable} and the $\sigma$-decomposition for smooth varieties.
\end{proof}

As a consequence of Proposition \ref{prop teissier movable}, we show the strict log concavity of the volume function $\vol$ on the cone of big and movable divisors.

\begin{prop} \label{strictlyconvexvol}
Let $X$ be a projective variety of dimension $n$. Then the volume function $\vol$ is strictly log concave on the cone of big and movable divisor classes.
\end{prop}

\begin{proof}
Since the big and movable cone is convex and since the derivative of $\vol$ is continuous, this follows immediately from Proposition \ref{prop teissier movable} and Theorem \ref{thrm teissier}.
\end{proof}

\subsection{The function $\mathfrak{M}$}

We now return to the study of the function $\mathfrak{M}$. We are in the situation:
\begin{align*}
\mathcal{C}=\Eff^1 (X), \quad f=\vol,\quad \mathcal{C}^*=\Mov_1 (X),\quad \mathcal{H}f= \mathfrak{M}.
\end{align*}
Note that $\mathcal{C}^*=\Mov_1 (X)$ follows from the main result of \cite{BDPP13}.

As preparation for using the polar transform theory, we recall the analytic properties of the volume function for divisors on smooth varieties.  By \cite{bfj09} the volume function on the pseudo-effective cone of divisors is differentiable on the big cone (with $D(L) = \langle L^{n-1} \rangle$).  In the notation of Section \ref{legendresection} the cone $\Eff^{1}(X)_{\vol}$ coincides with the big cone, so that $\vol$ is $+$-differentiable.  The volume function is $n$-concave, and is strictly $n$-concave on the big and movable cone by Proposition \ref{strictlyconvexvol}.  Furthermore, it admits a strong Zariski decomposition with respect to the movable cone of divisors using the $\sigma$-decomposition of \cite{Nak04} and Proposition \ref{strictlyconvexvol}.

\begin{rmk} \label{firstsingrmk}
Note that if $X$ is not smooth (or at least $\mathbb{Q}$-factorial), then it is unclear whether $\vol$ admits a Zariski decomposition structure with respect to the cone of movable divisors.  For this reason, we will focus on smooth varieties in this section.  See Remark \ref{singularcasermk} for more details.
\end{rmk}

Our first task is to understand the behaviour of $\mathfrak{M}$ on the boundary of the movable cone of curves.  Note that $\vol$ does not satisfy a sublinear condition, so that $\mathfrak{M}$ may not vanish on the boundary of $\Mov_{1}(X)$.

\begin{lem} \label{lem computing m}
Let $X$ be a smooth projective variety of dimension $n$ and let $\alpha$ be a movable curve class.  Then $\mathfrak{M}(\alpha) = 0$ if and only if $\alpha$ has vanishing intersection against a non-zero movable divisor class $L$.
\end{lem}

\begin{proof}
We first show that if there exists some nonzero movable divisor class $M$ such that $\alpha\cdot M=0$ then $\mathfrak{M}(\alpha) = 0$. Fix an ample divisor class $A$.  Note that $M+\epsilon A$ is big and movable for any $\epsilon>0$. Thus there exists some modification $\mu_{\epsilon}: Y_{\epsilon} \rightarrow X$ and an ample divisor class $A_\epsilon$ on $Y_{\epsilon}$ such that $M+\frac{\epsilon}{2} A=\mu_{\epsilon*} A_\epsilon$. So we can write $$M+\epsilon A=\mu_{\epsilon *}\left(A_\epsilon + \frac{\epsilon}{2} \mu_{\epsilon}^*A \right),$$
which implies
\begin{align*}
\vol(M+\epsilon A)&=\vol\left(\mu_{\epsilon*}\left(A_\epsilon + \frac{\epsilon}{2} \mu_{\epsilon}^*A\right) \right)\\
&\geq \vol \left(A_\epsilon + \frac{\epsilon}{2} \mu_{\epsilon}^*A \right)\\
&\geq n \left(\frac{\epsilon}{2} \mu_{\epsilon}^*A \right)^{n-1}\cdot A_\epsilon\\
&\geq c\epsilon^{n-1} A^{n-1}\cdot M.
\end{align*}
Consider the following intersection number
$$
\rho_\epsilon=\alpha \cdot \frac{M+\epsilon A}{\vol(M+\epsilon A)^{1/n}}.
$$
The above estimate shows that $\rho_\epsilon$ tends to zero as $\epsilon$ tends to zero, and so $\mathfrak{M}(\alpha)=0$.

Conversely, suppose that $\mathfrak{M}(\alpha) = 0$.  From the definition of $\mathfrak{M}(\alpha)$, we can take a sequence of big divisor classes $L_k$ with $\vol(L_k)=1$ such that $$\underset{k\rightarrow \infty}{\lim}(\alpha \cdot L_k)^\frac{n}{n-1} = \mathfrak{M}(\alpha).$$
Moreover, let $P_\sigma(L_k)$ be the positive part of $L_k$.  Then we have $\vol(P_\sigma(L_k))=1$ and $$\alpha\cdot P_\sigma(L_k) \leq \alpha\cdot L_k$$ since $\alpha$ is movable. Thus we can assume the sequence of big divisor classes $L_k$ is movable in the beginning.

Fix an ample divisor $A$ of volume $1$, and consider the classes $L_{k}/(A^{n-1} \cdot L_{k})$.  These lie in a compact slice of the movable cone, so they must have a non-zero movable accumulation point $L$, which without loss of generality we may assume is a limit.

Choose a modification $\mu_{\epsilon}: Y_{\epsilon} \to X$ and an ample divisor class $A_{\epsilon,k}$ on $Y$ such that
\begin{align*}
A_{\epsilon,k} \leq \mu_{\epsilon}^{*}L_{k}, \quad \vol(A_{\epsilon,k}) > \vol(L_{k}) - \epsilon
\end{align*}
Then
\begin{equation*}
L_{k} \cdot A^{n-1} \geq A_{\epsilon,k} \cdot \mu_{\epsilon}^{*}A^{n-1} \geq \vol(A_{\epsilon,k})^{1/n}
\end{equation*}
by the Khovanskii-Teissier inequality.  Taking a limit over all $\epsilon$, we find $L_{k} \cdot A^{n-1} \geq \vol(L_{k})^{1/n}$.  Thus
\begin{align*}
L \cdot \alpha  = \lim_{k \to \infty} \frac{L_{k} \cdot \alpha}{L_{k} \cdot A^{n-1}}
 \leq \mathfrak{M}(\alpha)^{n-1/n} = 0.
\end{align*}
\end{proof}

\begin{exmple}
Note that a movable curve class $\alpha$ with positive $\mathfrak{M}$ need not lie in the interior of the movable cone of curves.  A simple example is when $X$ is the blow-up of $\mathbb{P}^{2}$ at one point, $H$ denotes the pullback of the hyperplane class.  For surfaces the functions $\mathfrak{M}$ and $\vol$ coincide, so $\mathfrak{M}(H) = 1$ even though $H$ is on the boundary of $\Mov_{1}(X) = \Nef^{1}(X)$.

It is also possible for a big movable curve class $\alpha$ to have $\mathfrak{M}(\alpha) = 0$.  This occurs for the projective bundle $X = \mathbb{P}_{\mathbb{P}^{1}}(\mathcal{O} \oplus \mathcal{O} \oplus \mathcal{O}(-1))$. There are two natural divisor classes on $X$: the class $f$ of the fibers of the projective bundle and the class $\xi$ of the sheaf $\mathcal{O}_{X/\mathbb{P}^{1}}(1)$.  Using for example \cite[Theorem 1.1]{fulger11} and \cite[Proposition 7.1]{fl14}, one sees that $f$ and $\xi$ generate the algebraic cohomology classes with the relations
$f^{2} = 0$, $\xi^{2}f = -\xi^{3} = 1$ and that $\Mov^{1}(X) = \langle f, \xi \rangle$ and $\Mov_{1}(X) = \langle \xi f, \xi^{2}+\xi f \rangle$.
We see that the big and movable curve class $\xi^{2} + \xi f$ has vanishing intersection against the movable divisor $\xi$ so that $\mathfrak{M}(\xi^{2}+\xi f) = 0$ by Lemma \ref{lem computing m}.
\end{exmple}

\begin{rmk}
Another perspective on Lemma \ref{lem computing m} is provided by the numerical dimension of \cite{Nak04} and \cite{Bou04}.  We recall from \cite{lehmann13} the fact that on a smooth variety the following conditions are equivalent for a class $L \in \Eff^{1}(X)$.  (They both correspond to the non-vanishing of the numerical dimension.)
\begin{itemize}
\item Fix an ample divisor class $A$. For any big class $D$, there is a positive constant $C$ such that $Ct^{n-1} < \vol(L + tA)$ for all $t>0$.
\item $P_{\sigma}(L) \neq 0$.
\end{itemize}
In particular, this implies that $\vol$ satisfies the sublinear boundary condition of order $n-1/n$ when restricted to the movable cone, and this fact can be used in the previous proof.  A variant of this statement in characteristic $p$ is proved by \cite{chms14}.
\end{rmk}

In many ways it is most natural to define $\mathfrak{M}$ using the movable cone of divisors instead of the pseudo-effective cone of divisors.  Conceptually, this coheres with the fact that the polar transform can be calculated using the positive part of a Zariski decomposition. Recall that the positive part $P_{\sigma}(L)$ of a pseudo-effective divisor $L$ has $P_{\sigma}(L) \preceq L$ and $\vol(P_{\sigma}(L)) = \vol(L)$.  Arguing as in Lemma \ref{lem computing m} by taking positive parts, we see that for any $\alpha \in \Mov_{1}(X)$ we have
\begin{align*}
\mathfrak{M}(\alpha)= \inf_{D \textrm{ big and movable}} \left( \frac{D \cdot \alpha}{\vol(D)^{1/n}} \right)^{n/n-1}.
\end{align*}
Thus for $X$ smooth it is perhaps better to consider the following polar transform:
\begin{align*}
\mathcal{C}=\Mov^1 (X), \quad f=\vol,\quad \mathcal{C}^*=\Mov^1 (X)^*,\quad \mathcal{H}f= \mathfrak{M}'.
\end{align*}
Since $\vol$ satisfies a sublinear condition on $\Mov^1 (X)$, the function $\mathfrak{M}'$ is strictly positive exactly in $\Mov^1 (X)^{*\circ}$ and extends to a continuous function over $N_1(X)$.  The relationship between the two functions is given by
\begin{equation*}
\mathfrak{M}'|_{\Mov_{1}(X)} = \mathfrak{M};
\end{equation*}
this follows immediately from the description for $\mathfrak{M}$ earlier in this paragraph.  In fact by Theorem \ref{strong zariski equivalence} $\mathfrak{M}'$ admits a strong Zariski decomposition.  Using the interpretation of positive parts via derivatives as in Theorem \ref{strong zariski equivalence}, the results of \cite{bfj09} and \cite{lm09} show that the positive parts for the Zariski decomposition of $\mathfrak{M}'$ lie in $\Mov_{1}(X)$.  In this way one can think of $\mathfrak{M}$ as the ``Zariski projection'' of $\mathfrak{M}'$.

Note one important consequence of this perspective: Lemma \ref{lem computing m} shows that the subcone of $\Mov_{1}(X)$ where $\mathfrak{M}$ is positive lies in the interior of $\Mov^{1}(X)^{*}$.  Thus this region agrees with $\Mov_{1}(X)_{\mathfrak{M}}$ and $\mathfrak{M}$ extends to a differentiable function on an open set containing this cone by applying Theorem \ref{strong zariski equivalence}.  In particular $\mathfrak{M}$ is $+$-differentiable and continuous on $\Mov_{1}(X)$.

We next prove a refined structure of the movable cone of curves.  Recall that by \cite{BDPP13} the movable cone of curves $\Mov_1(X)$ is generated by the $(n-1)$-self positive products of big divisors.  In other words, any curve class in the interior of $\Mov_{1}(X)$ is a \emph{convex combination} of such positive products.  We show that $\Mov_1(X)$ actually coincides with the closure of such products (which naturally form a cone).

\begin{thrm}
\label{thm mov-intersection after}
Let $X$ be a smooth projective variety of dimension $n$. Then any movable curve class $\alpha$ with $\mathfrak{M}(\alpha) > 0$ has the form
\begin{equation*}
\alpha=\langle L_{\alpha}^{n-1}\rangle
\end{equation*}
for a unique big and movable divisor class $L_{\alpha}$.  We then have $\mathfrak{M}(\alpha) = \vol(L_{\alpha})$ and any big and movable divisor computing $\mathfrak{M}(\alpha)$ is proportional to $L_{\alpha}$.
\end{thrm}

\begin{proof}
Applying Theorem \ref{strong zariski equivalence} to $\mathfrak{M}'$, we get
$$\alpha =D(L_\alpha)+n_\alpha$$
where $L_\alpha$ is a big movable class computing $\mathfrak{M}(\alpha)$ and $n_\alpha \in \Mov^1 (X)^*$. As $D$ is the differential of $\vol^{1/n}$ on big and movable divisor classes, we have
$D(L_\alpha)=\langle L_{\alpha} ^{n-1}\rangle$. Note that $\mathfrak{M}(\alpha) = \langle L_{\alpha}^{n-1} \rangle \cdot L_{\alpha} = \vol(L_{\alpha})$.

To finish the proof, we observe that $n_\alpha \in \Mov_1 (X)$. This follows since $\alpha$ is movable: by the definition of $L_\alpha$, for any pseudo-effective divisor class $E$ and $t\geq 0$ we have
\begin{align*}
 \frac{\alpha\cdot L_\alpha}{\vol(L_\alpha)^{1/n}} \leq \frac{\alpha\cdot P_\sigma (L_\alpha+tE)}{\vol( L_\alpha+tE)^{1/n}}
 \leq \frac{\alpha\cdot (L_\alpha+tE)}{\vol( L_\alpha+tE)^{1/n}}
\end{align*}
with equality at $t=0$. This then implies
\begin{align*}
n_\alpha \cdot E \geq 0.
\end{align*}
Thus $n_\alpha \in \Mov_1 (X)$.
Intersecting against $L_\alpha$, we have $n_\alpha \cdot L_\alpha =0$. This shows $n_\alpha =0$ because $L_\alpha$ is an interior point of $\Eff^1 (X)$ and $\Eff^1 (X)^* =\Mov_1 (X)$.

So we have
$$\alpha =D(L_\alpha)=\langle L_{\alpha} ^{n-1}\rangle.$$
Finally, the uniqueness follows from Corollary \ref{cor mov-intersection after}.
\end{proof}

We note in passing that we immediately obtain:

\begin{cor}
Let $X$ be a projective variety of dimension $n$.  Then the rays spanned by classes of irreducible curves which deform to cover $X$ are dense in $\Mov_{1}(X)$.
\end{cor}

\begin{proof}
It suffices to prove this on a resolution of $X$.  By Theorem \ref{thm mov-intersection after} it suffices to show that any class of the form $\langle L^{n-1} \rangle$ for a big divisor $L$ is a limit of rescalings of classes of irreducible curves which deform to cover $X$.  Indeed, we may even assume that $L$ is a $\mathbb{Q}$-Cartier divisor.  Then the positive product is a limit of the pushforward of complete intersections of ample divisors on birational models, whence the result.
\end{proof}

We can also describe the boundary of $\Mov_{1}(X)$, in combination with Lemma \ref{lem computing m}.

\begin{cor} \label{cor boundary to boundary}
Let $X$ be a smooth projective variety of dimension $n$.  Let $\alpha$ be a movable class with $\mathfrak{M}(\alpha) > 0$ and let $L_\alpha$ be the unique big movable divisor whose positive product is $\alpha$.  Then $\alpha$ is on the boundary of $\Mov_{1}(X)$ if and only if $L_\alpha$ is on the boundary of $\Mov^{1}(X)$.
\end{cor}

\begin{proof}
Note that $\alpha$ is on the boundary of $\Mov_{1}(X)$ if and only if it has vanishing intersection against a class $D$ lying on an extremal ray of $\Eff^{1}(X)$.  Lemma \ref{lem computing m} shows that in this case $D$ is not movable, so by \cite[Chapter III.1]{Nak04} $D$ is (after rescaling) the class of an integral divisor on $X$ which we continue to call $D$.  By \cite[Proposition 4.8 and Theorem 4.9]{bfj09}, the equation $\langle L_\alpha ^{n-1} \rangle \cdot D = 0$ holds if and only if $D \in \mathbb{B}_{+}(L_\alpha)$.  Altogether, we see that $\alpha$ is on the boundary of $\Mov_{1}(X)$ if and only if $L_\alpha$ is on the boundary of $\Mov^{1}(X)$.
\end{proof}

Arguing using abstract properties of polar transforms just as in \cite{lehmannxiao2015a}, the good analytic properties of the volume function for divisors imply most of the other analytic properties of $\mathfrak{M}$.

\begin{thrm} (see Theorem \ref{thrm positive part continuity} and compare with \cite[Theorem 5.6]{lehmannxiao2015a})
\label{thm continuity big movable}
Let $X$ be a smooth projective variety of dimension $n$.  For any movable curve class $\alpha$ with $\mathfrak{M}(\alpha)>0$, let $L_{\alpha}$ denote the unique big and movable divisor class satisfying $\langle L_{\alpha}^{n-1} \rangle = \alpha$.   As we vary $\alpha$ in $\Mov_1 (X)_{\mathfrak{M}}$, $L_\alpha$ depends continuously on $\alpha$.
\end{thrm}

\begin{thrm} (
compare with \cite[Theorem 5.11]{lehmannxiao2015a})
\label{thm direvative movable}
Let $X$ be a smooth projective variety of dimension $n$.
For a curve class $\alpha=\langle L_\alpha ^{n-1}\rangle$ in $\Mov_1 (X)_{\mathfrak{M}}$ and for an arbitrary curve class $\beta\in N_1(X)$ we have
\begin{align*}
\left. \frac{d}{dt} \right|_{t=0} \mathfrak{M}(\alpha+t \beta)=\frac{n}{n-1}P_\sigma (L_\alpha) \cdot \beta.
\end{align*}
\end{thrm}

\begin{thrm} (see Theorem \ref{strong zariski equivalence} and compare with \cite[Theorem 5.10]{lehmannxiao2015a})
\label{thm strict concavity M hatvol}
Let $X$ be a smooth projective variety of dimension $n$.  Let $\alpha_1, \alpha_2$ be two big and movable curve classes in $\Mov_1 (X)_{\mathfrak{M}}$.  Then $$\mathfrak{M}(\alpha_1+ \alpha_2)^{n-1/n}\geq \mathfrak{M}(\alpha_1)^{n-1/n}+ \mathfrak{M}(\alpha_2)^{n-1/n}$$ with equality if and only if $\alpha_1$ and $\alpha_2$ are proportional.
\end{thrm}

Another application of the results in this section is the Morse-type bigness criterion for movable curve classes, which is slightly different from \cite[Theorem 5.18]{lehmannxiao2015a}.

\begin{thrm}
\label{thm movable curve morse}
Let $X$ be a smooth projective variety of dimension $n$. Let $\alpha, \beta$ be two curve classes lying in $\Mov_1 (X)_{\mathfrak{M}}$.
Write $\alpha=\langle L_\alpha^{n-1} \rangle$ and $\beta=\langle L_\beta ^{n-1} \rangle$ for the unique big and movable divisor classes $L_\alpha, L_\beta$ given by Theorem \ref{thm mov-intersection after}. Then we have
\begin{align*}
\mathfrak{M}(\alpha-\beta)^{n-1/n}
&\geq (\mathfrak{M}(\alpha)-n L_\alpha \cdot \beta)\cdot \mathfrak{M}(\alpha)^{-1/n}\\
&=(\vol(L_\alpha)-n L_\alpha \cdot \beta)\cdot \vol(L_\alpha)^{-1/n}.
\end{align*}
In particular, we have
\begin{align*}
\mathfrak{M}(\alpha-\beta)\geq \vol(L_\alpha)-\frac{n^2}{n-1} L_\alpha \cdot \beta
\end{align*}
and the curve class $\alpha-\beta$ is big whenever $\mathfrak{M}(\alpha)-n L_\alpha \cdot \beta >0$.
\end{thrm}

\begin{proof}
By \cite[Section 4.2]{lehmannxiao2015a} it suffices to prove a Morse-type bigness criterion for the difference of two movable divisor classes. So we need to prove $L-M$ is big whenever
$$\langle L^n\rangle -n \langle L^{n-1}\rangle\cdot M >0.$$
This is proved (in the K\"ahler setting) in \cite[Theorem 1.1]{xiao2014movable}.
\end{proof}

\begin{rmk}
We remark that we can not extend this Morse-type criterion from big and movable divisors to arbitrary pseudo-effective divisor classes.  A very simple construction provides the counter examples, e.g.~the blow up of $\mathbb{P}^2$ (see \cite[Example 3.8]{Tra95morse}).
\end{rmk}

Combining Theorem \ref{thm mov-intersection after} and Theorem \ref{thm continuity big movable}, we obtain:

\begin{cor} \label{coneisocor}
Let $X$ be a smooth projective variety of dimension $n$.  Then
$$\Phi: \Mov^1(X)_{\vol} \rightarrow \Mov_1 (X)_{\mathfrak{M}},\qquad L \mapsto \langle L^{n-1}\rangle $$
is a homeomorphism.
\end{cor}

\begin{rmk}
Corollary \ref{coneisocor} gives a systematic way of translating between ``chamber decompositions'' on $\Mov_{1}(X)$ and $\Mov^{1}(X)$.  This relationship could be exploited to elucidate the geometry underlying chamber decompositions.

One potential application is in the study of stability conditions.  For example, \cite{Neu10} studies a decomposition of $\Mov_1(X)$ into chambers defining different Harder-Narasimhan filtrations of the tangent bundle of $X$ with respect to movable curves.  Let $\alpha$ be a movable curve class. Denote by $\HNF(\alpha, TX)$ the Harder-Narasimhan filtration of the tangent bundle with respect to the class $\alpha$. Then we have the following ``destabilizing chambers":
$$
\Delta_\alpha:=\{ \beta\in \Mov_1(X)| \HNF(\beta, TX)=\HNF(\alpha, TX) \}.
$$
By \cite[Theorem 3.3.4, Proposition 3.3.5]{Neu10}, the destabilizing chambers are pairwise disjoint and provide a decomposition of the movable cone $\Mov_1(X)$. Moreover, the decomposition is locally finite in $\Mov_1(X)^\circ$ and the destabilizing chambers are convex cones whose closures are locally polyhedral in $\Mov_1(X)^\circ$. In particular, if $\Mov_1(X)$ is polyhedral, then the chamber structure is finite.

For Fano threefolds, \cite{Neu10} shows that the destabilizing subsheaves are all relative tangent sheaves of some Mori fibration on $X$.  See also \cite{Keb13} for potential applications of this analysis.
It would be interesting to study whether the induced filtrations on $TX$ are related to the geometry of the movable divisors $L$ in the $\Phi$-inverse of the corresponding chamber of $\Mov_{1}(X)$.
\end{rmk}

\begin{rmk} \label{singularcasermk}
Modified versions of many of the results in this section hold for singular varieties as well (see Remark \ref{firstsingrmk}).  For example, by similar arguments we can see that any element in the interior of $\Mov_{1}(X)$ is the positive product of some big divisor class regardless of singularities.  Conversely, whenever $\mathfrak{M}$ is $+$-differentiable we obtain a Zariski decomposition structure for $\vol$ by Theorem \ref{strong zariski equivalence}.
\end{rmk}

\begin{rmk}
All the results above extend to smooth varieties over algebraically closed fields.  However, for compact K\"ahler manifolds some results rely on Demailly's conjecture on the transcendental holomorphic Morse-type inequality, or equivalently, on the extension of the results of \cite{bfj09} to the K\"ahler setting.  Since the results of \cite{bfj09} are used in an essential way in the proof of Theorems \ref{thm mov-intersection after} and \ref{lemma psefconstantbound} (via the proof of \cite[Proposition 5.3]{fl14}), the only statement in this section which extends unconditionally to the K\"ahler setting is Lemma \ref{lem computing m}.
However, these conjectures are known if $X$ is a compact hyperk\"ahler manifold (see \cite[Theorem 10.12]{BDPP13}), so all of our results extend to compact hyperk\"ahler manifolds.  
\end{rmk}

\section{Toric varieties}
\label{toric section}

We study the function $\mathfrak{M}$ on toric varieties, showing that it can be interpreted by the underlying special structures.
In this section, $X$ will denote a simplicial projective toric variety of dimension $n$.  In terms of notation, $X$ will be defined by a fan $\Sigma$ in a lattice $N$ with dual lattice $M$.  We let $\{ v_{i} \}$ denote the primitive generators of the rays of $\Sigma$ and $\{ D_{i} \}$ denote the corresponding classes of $T$-divisors.  Our goal is to interpret the properties of the function $\mathfrak{M}$ in terms of toric geometry.

\subsection{Positive product on toric varieties}
Suppose that $L$ is a big movable divisor class on the toric variety $X$.  Then $L$ naturally defines a (non-lattice) polytope $Q_{L}$: if we choose an expression $L = \sum a_{i}D_{i}$, then
\begin{equation*}
Q_{L} = \{ u \in M_{\mathbb{R}} |  \langle u,v_{i} \rangle + a_{i} \geq 0\}
\end{equation*}
and changing the choice of representative corresponds to a translation of $Q_{L}$.
Conversely, suppose that $Q$ is a full-dimensional polytope such that the unit normals to the facets of $Q$ form a subset of the rays of $\Sigma$.  Then $Q$ uniquely determines a big movable divisor class $L_{Q}$ on $X$.  The divisors in the interior of the movable cone correspond to those polytopes whose facet normals coincide with the rays of $\Sigma$.

Given polytopes $Q_{1},\ldots,Q_{n}$, let $V(Q_{1},\ldots,Q_{n})$ denote the mixed volume of the polytopes.  \cite{bfj09} explains that the positive product of big movable divisors $L_{1},\ldots,L_{n}$ can be interpreted via the mixed volume of the corresponding polytopes:
\begin{equation*}
\langle L_{1} \cdot \ldots \cdot L_{n} \rangle = n! V(Q_{1},\ldots,Q_{n}).
\end{equation*}

\subsection{The function $\mathfrak{M}$}

In this section we use a theorem of Minkowski to describe the function $\mathfrak{M}$.  We thank J.~Huh for a conversation working out this picture.

Recall that a class $\alpha \in \Mov_{1}(X)$ defines a non-negative Minkowski weight on the rays of the fan $\Sigma$ -- that is, an assignment of a positive real number $t_{i}$ to each vector $v_{i}$ such that $\sum t_{i}v_{i} = 0$.  From now on we will identify $\alpha$ with its Minkowski weight.  We will need to identify which movable curve classes are positive along a set of rays which span $\mathbb{R}^{n}$.

\begin{lem}
Suppose $\alpha \in \Mov_{1}(X)$ satisfies $\mathfrak{M}(\alpha) > 0$.  Then $\alpha$ is positive along a spanning set of rays of $\Sigma$.
\end{lem}

We will soon see that the converse is also true in Theorem \ref{toricmfunction}.

\begin{proof}
Suppose that there is a hyperplane $V$ which contains every ray of $\Sigma$ along which $\alpha$ is positive.  Since $X$ is projective, $\Sigma$ has rays on both sides of $V$.  Let $D$ be the effective divisor consisting of the sum over all the primitive generators of rays of $\Sigma$ not contained in $V$.  It is clear that the polytope defined by $D$ has non-zero projection onto the subspace spanned by $V^{\perp}$, and in particular, that the polytope defined by $D$ is non-zero.  Thus the asymptotic growth of sections of $mD$ is at least linear in $m$, so $P_{\sigma}(D) \neq 0$ and  $\alpha$ has vanishing intersection against a non-zero movable divisor.  Lemma \ref{lem computing m} shows that $\mathfrak{M}(\alpha) = 0$.
\end{proof}

Minkowski's theorem asserts the following.  Suppose that $u_{1},\ldots,u_{s}$ are unit vectors which span $\mathbb{R}^{n}$ and that $r_{1},\ldots,r_{s}$ are positive real numbers.  Then there exists a polytope $P$ with unit normals $u_{1},\ldots,u_{s}$ and with corresponding facet volumes $r_{1},\ldots,r_{s}$ if and only if the $u_{i}$ satisfy 
\begin{equation*}
r_{1}u_{1} + \ldots + r_{s}u_{s} = 0.
\end{equation*}
Moreover, the resulting polytope is unique up to translation.  (See \cite{klain04} for a proof which is compatible with the results below.)  If a vector $u$ is a unit normal to a facet of $P$, we will use the notation $\vol(P^{u})$ to denote the volume of the facet corresponding to $u$.

If $\alpha$ is positive on a spanning set of rays, then it canonically defines a polytope (up to translation) via Minkowski's theorem by choosing the vectors $u_{i}$ to be the unit vectors in the directions $v_{i}$ and assigning to each the constant
\begin{equation*}
r_{i} = \frac{t_{i}|v_{i}|}{(n-1)!}.
\end{equation*}
Note that this is the natural choice of volume for the corresponding facet, as it accounts for:
\begin{itemize}
\item the discrepancy in length between $u_{i}$ and $v_{i}$, and
\item the factor $\frac{1}{(n-1)!}$ relating the volume of an $(n-1)$-simplex to the determinant of its edge vectors.
\end{itemize}
We denote the corresponding polytope in $M_{\mathbb{R}}$ defined by the theorem of Minkowski by $P_{\alpha}$.

\begin{thrm} \label{toricmfunction}
Suppose $\alpha$ is a movable curve class which is positive on a spanning set of rays and let $P_{\alpha}$ be the corresponding polytope.  Then
\begin{equation*}
\mathfrak{M}(\alpha) = n! \vol(P_{\alpha}).
\end{equation*}
Furthermore, the big movable divisor $L_{\alpha}$ corresponding to the polytope $P_{\alpha}$ satisfies $\langle L_{\alpha}^{n-1} \rangle = \alpha$.
\end{thrm}

\begin{proof}
Let $L \in \Mov^{1}(X)$ be a big movable divisor class and denote the corresponding polytope by $Q_{L}$.  We claim that the intersection number can be interpreted as a mixed volume:
\begin{equation*}
L \cdot \alpha = n! V(P_{\alpha}^{n-1},Q_{L}).
\end{equation*}
To see this, define for a compact convex set $K$ the function $h_{K}(u) = \sup_{v \in K} \{v \cdot u\}$.  Using \cite[Equation (5)]{klain04}
\begin{align*}
V(P_{\alpha}^{n-1},Q_{L}) & = \frac{1}{n} \sum_{u \textrm{ a facet of }P_{\alpha}+Q_{L}} h_{Q_{L}}(u) \vol(P_{\alpha}^{u}) \\
& = \frac{1}{n} \sum_{\textrm{rays }v_{i}} \left( \frac{a_{i}}{|v_{i}|} \right) \left( \frac{t_{i}|v_{i}|}{(n-1)!} \right) \\
& = \frac{1}{n!} \sum_{\textrm{rays }v_{i}} a_{i}t_{i} = \frac{1}{n!}L \cdot \alpha.
\end{align*}
Note that we actually have equality in the second line because $L$ is big and movable.  Recall that by the Brunn-Minkowski inequality
\begin{equation*}
V(P_{\alpha}^{n-1},Q_{L}) \geq \vol(P_{\alpha})^{n-1/n} \vol(Q_{L})^{1/n}
\end{equation*}
with equality only when $P_{\alpha}$ and $Q_{L}$ are homothetic.  Thus
\begin{align*}
\mathfrak{M}(\alpha) & = \inf_{L \textrm{ big movable class}} \left( \frac{L \cdot \alpha}{\vol(L)^{1/n}} \right)^{n/n-1} \\
& =  \inf_{L \textrm{ big movable class}}\left(  \frac{n! V(P_{\alpha}^{n-1},Q_{L})}{n!^{1/n} \vol(Q_{L})^{1/n}} \right)^{n/n-1} \\
& \geq n! \vol(P_{\alpha}).
\end{align*}
Furthermore, the equality is achieved for divisors $L$ whose polytope is homothetic to $P_{\alpha}$, showing the computation of $\mathfrak{M}(\alpha)$.  Furthermore, since the divisor $L_{\alpha}$ defined by the polytope computes $\mathfrak{M}(\alpha)$ we see that $\langle L_{\alpha}^{n-1} \rangle$ is proportional to $\alpha$.  By computing $\mathfrak{M}$ we deduce the equality:
\begin{equation*}
\mathfrak{M}(\langle L_{\alpha}^{n-1} \rangle) = \vol(L) = n! \vol(P_{\alpha}) = \mathfrak{M}(\alpha).
\end{equation*}
\end{proof}

The previous result shows:

\begin{cor} \label{toriccicor}
Let $\alpha$ be a curve class in $\Mov_{1}(X)_{\mathfrak{M}}$.  Then $\alpha \in \CI_{1}(X)$ if and only if the normal fan to the corresponding polytope $P_{\alpha}$ is refined by $\Sigma$.  In this case we have
\begin{equation*}
\widehat{\vol}(\alpha) = n! \vol(P_{\alpha}).
\end{equation*}
\end{cor}

\begin{proof}
By the uniqueness in Theorem \ref{thm mov-intersection after}, $\alpha \in \CI_{1}(X)$ if and only if the corresponding divisor $L_{\alpha}$ as in Theorem \ref{toricmfunction} is big and nef.
\end{proof}

\section{Comparing the complete intersection cone and the movable cone}
\label{ci and mov section}

Consider the functions $\widehat{\vol}$ and $\mathfrak{M}$ on the movable cone of curves $\Mov_1(X)$.  By their definitions we always have $\widehat{\vol}\geq \mathfrak{M}$ on the movable cone, and \cite[Remark 3.1]{xiao15} asks whether one can characterize when equality holds. In this section we show:

\begin{thrm}
\label{thm solution to xiao}
Let $X$ be a smooth projective variety of dimension $n$ and let $\alpha$ be a big and movable class. Then $\widehat{\vol}(\alpha)> \mathfrak{M}(\alpha)$ if and only if $\alpha \notin \CI_1(X)$.
\end{thrm}

Thus $\widehat{\vol}$ and $\mathfrak{M}$ can be used to distinguish whether a big movable curve class lies in $\CI_1(X)$ or not. This result is important in Section \ref{section mobility}.

\begin{proof}
If $\alpha = B^{n-1}$ is a complete intersection class, then $\widehat{\vol}(\alpha) = \vol(B) = \mathfrak{M}(\alpha)$.  By continuity the equality holds true for any big curve class in $\CI_{1}(X)$.

Conversely, suppose that $\alpha$ is not in the complete intersection cone.  The claim is clearly true if $\mathfrak{M}(\alpha) = 0$, so by Theorem \ref{thm mov-intersection after} it suffices to consider the case when there is a big and movable divisor class $L$ such that $\alpha = \langle L^{n-1} \rangle$.  Note that $L$ can not be big and nef since $\alpha \notin \CI_1(X)$.

We prove $\widehat{\vol}(\alpha)> \mathfrak{M}(\alpha)$ by contradiction. First, by the definition of $\widehat{\vol}$ we always have $$\widehat{\vol}(\langle L^{n-1} \rangle)\geq \mathfrak{M}(\langle L^{n-1} \rangle) = \vol(L).$$
 Suppose $\widehat{\vol}(\langle L^{n-1} \rangle) =\vol(L)$. For convenience, we assume $\vol(L)=1$.  By rescaling the positive part of a Zariski decomposition, we find a big and nef divisor class $B$ with $\vol(B)=1$ such that $\widehat{\vol}(\langle L^{n-1} \rangle) = (\langle L^{n-1} \rangle \cdot B)^{n/n-1}$. For the divisor class $B$ we get
$$
\langle L^{n-1} \rangle\cdot B= 1 =  \vol(L)^{n-1/n}\vol(B)^{1/n}.
$$
By Proposition \ref{prop teissier movable}, this implies $L$ and $B$ are proportional which contradicts the non-nefness of $L$. Thus we must have $\widehat{\vol}(\langle L^{n-1} \rangle) > \vol(L) = \mathfrak{M}(\langle L^{n-1} \rangle)$.
\end{proof}

We also obtain:
\begin{prop}
\label{prop hatvol birational}
Let $X$ be a smooth projective variety of dimension $n$ and let $\alpha$ be a big and movable curve class. Assume that $\widehat{\vol}(\phi^* \alpha)=\widehat{\vol}(\alpha)$ for any birational morphism $\phi$.  Then $\alpha \in \CI_1(X)$.
\end{prop}

\begin{proof}
We first consider the case when $\mathfrak{M}(\alpha) > 0$.  Let $L$ be a big movable divisor class satisfying $\langle L^{n-1} \rangle = \alpha$.  Choose a sequence of birational maps $\phi_{\epsilon}: Y_{\epsilon} \to X$ and ample divisor classes $A_{\epsilon}$ on $Y_{\epsilon}$ defining an $\epsilon$-Fujita approximation for $L$.  Then $\vol(L) \geq \vol(A_{\epsilon}) > \vol(L) -\epsilon$ and the classes $\phi_{\epsilon *} A_{\epsilon}$ limit to $L$. Note that
$A_{\epsilon} \cdot \phi_{\epsilon}^{*}\alpha = \phi_{\epsilon *}A_{\epsilon} \cdot \alpha$.
This implies that for any $\epsilon > 0$ we have
\begin{equation*}
\widehat{\vol}(\alpha) = \widehat{\vol}(\phi_{\epsilon}^{*}\alpha) \leq \frac{(\alpha \cdot \phi_{\epsilon *} A_{\epsilon})^{n/n-1}}{\vol(L)^{1/n-1}}.
\end{equation*}
As $\epsilon$ shrinks the right hand side approaches $\vol(L) = \mathfrak{M}(\alpha)$, and we conclude by Theorem \ref{thm solution to xiao}.

Next we consider the case when $\mathfrak{M}(\alpha) = 0$. Choose a class $\xi$ in the interior of $\Mov_{1}(X)$ and consider the classes $\alpha + \delta \xi$ for $\delta >0$.  The argument above shows that for any $\epsilon > 0$, there is a birational model $\phi_{\epsilon}: Y_{\epsilon} \to X$ such that $$\widehat{\vol}(\phi_{\epsilon}^{*}(\alpha + \delta \xi)) < \mathfrak{M}(\alpha + \delta \xi) + \epsilon.$$
  But we also have $\widehat{\vol}(\phi_{\epsilon}^{*} \alpha) \leq \widehat{\vol}(\phi_{\epsilon}^{*}(\alpha + \delta \xi))$ since the pullback of the nef curve class $\delta \xi$ is pseudo-effective.  Taking limits as $\epsilon \to 0$, $\delta \to 0$, we see that we can make the volume of the pullback of $\alpha$ arbitrarily small, a contradiction to the assumption and the bigness of $\alpha$.
\end{proof}

Let $L$ be a big divisor class and let $\alpha=\langle L^{n-1}\rangle$ be the corresponding big movable curve class. From the Zariski decomposition $\alpha=B^{n-1}+\gamma$, we get a ``canonical" map $\pi$ from a big divisor class to a big and nef divisor class, that is, $\pi(L):=B$. Note that the map $\pi$ is continuous and satisfies $\pi^2=\pi$. It is natural to ask whether we can compare $L$ and $B$. However, if $P_\sigma (L)$ is not nef then $L$ and $B$ can not be compared:
\begin{itemize}
\item if $L\succeq B$ then we have $\vol(L)\geq \vol(B)$ which contradicts with
    Theorem \ref{thm solution to xiao};
\item if $L\preceq B$ then we have $\langle L^{n-1}\rangle \preceq B^{n-1}$ which contradicts with $\gamma \neq 0$.
\end{itemize}
If we modify the map $\pi$ a little bit, we can always get a ``canonical" nef divisor class lying below the big divisor class.

\begin{thrm}
  \label{thm big to nef}
Let $X$ be a smooth projective variety of dimension $n$, and let $\alpha$ be a big movable curve class. Let $L$ be a big divisor class such that $\alpha=\langle L^{n-1}\rangle$, and let $\alpha=B^{n-1}+ \gamma$ be the Zariski decomposition of $\alpha$. Define the map $\widehat{\pi}$ from the cone of big divisor classes to the cone of big and nef divisor classes as
$$\widehat{\pi}(L):= \left(1- \left(1-\frac{\mathfrak{M}(\alpha)}{\widehat{\vol}(\alpha)}
  \right)^{1/n} \right) B.$$
Then $\widehat{\pi}$ is a surjective continuous map satisfying $L\succeq \widehat{\pi}(L)$ and $\widehat{\pi}^2=\widehat{\pi}$.
\end{thrm}

\begin{proof}
It is clear if $L$ is nef then we have $\widehat{\pi}(L)=L$, and this implies $\widehat{\pi}$ is surjective and $\widehat{\pi}^2=\widehat{\pi}$. By Theorem \ref{thm curve decomposition} and Theorem \ref{thm continuity big movable}, we get the continuity of $\widehat{\pi}$. So we only need to verify $L\succeq \widehat{\pi}(L)$.
And this follows from the Diskant inequality for big and movable divisor classes.

Let $s$ be the largest real number such that $L\succeq s B$. By the properties of $\sigma$-decompositions, $s$ is also the largest real number such that $P_\sigma (L)\succeq s B$.
First, observe that $s\leq 1$ since $$\vol(L)=\mathfrak{M}(\alpha)\leq \widehat{\vol}(\alpha)=\vol(B).$$
Applying the Diskant inequality to $P_\sigma (L)$ and $B$, we have
\begin{align*}
  (\langle P_\sigma (L) ^{n-1}\rangle\cdot B)^{n/n-1}-&\vol(L)\vol(B)^{1/n-1}\\
   &\geq ((\langle P_\sigma (L) ^{n-1}\rangle\cdot B)^{1/n-1}- s \vol(B)^{1/n-1})^n.
\end{align*}
Note that $\widehat{\vol}(\alpha)=(\langle P_\sigma (L) ^{n-1}\rangle\cdot \frac{B}{\vol(B)^{1/n}})^{n/n-1}$ and $\mathfrak{M}(\alpha)=\vol(L)$. The above inequality implies
  $$
  s \geq 1-\left(1-\frac{\mathfrak{M}(\alpha)}{\widehat{\vol}(\alpha)}\right)^{1/n},
  $$
which yields the desired relation $L\succeq \widehat{\pi}(L)$.
\end{proof}

\begin{exmple}
\label{mdsexample}
Let $X$ be a Mori Dream Space.  Recall that a small $\mathbb{Q}$-factorial modification (henceforth SQM) $\phi: X \dashrightarrow X'$ is a birational contraction (i.e.~does not extract any divisors) defined in codimension $1$ such that $X'$ is projective $\mathbb{Q}$-factorial.   \cite{hk00} shows that for any SQM the strict transform defines an isomorphism $\phi_{*}: N^{1}(X) \to N^{1}(X')$ which preserves the pseudo-effective and movable cones of divisors.  (More generally, any birational contraction induces an injective pullback $\phi^{*}: N^{1}(X') \to N^{1}(X)$ and dually a surjection $\phi_{*}: N_{1}(X) \to N_{1}(X')$.)  The SQM structure induces a chamber decomposition of the pseudo-effective and movable cones of divisors.

One would like to see a ``dual picture'' in $N_{1}(X)$ of this chamber decomposition.  However, it does not seem interesting to simply dualize the divisor decomposition: the resulting cones are no longer pseudo-effective and are described as intersections instead of unions.  Motivated by the Zariski decomposition for curves, we define a chamber structure on the movable cone of curves as a union of the complete intersection cones on SQMs.

Note that for each SQM we obtain by duality an isomorphism $\phi_{*}: N_{1}(X) \to N_{1}(X')$ which preserves the movable cone of curves.  We claim that the strict transforms of the various complete intersection cones define a chamber structure on $\Mov_{1}(X)$.  More precisely, given any birational contraction $\phi: X \dashrightarrow X'$ with $X'$ normal projective, define
\begin{equation*}
\CI^{\circ}_{\phi} := \bigcup_{A \textrm{ ample on }X'} \langle \phi^{*}A^{n-1} \rangle.
\end{equation*}
Then
\begin{itemize}
\item $\Mov_{1}(X)$ is the union over all SQMs $\phi: X \dashrightarrow X'$ of $\overline{\CI^{\circ}_{\phi}} = \phi^{-1}_{*}\CI_{1}(X')$, and the interiors of the $\overline{\CI^{\circ}_{\phi}}$ are disjoint.
\item The set of classes in $\Mov_{1}(X)_{\mathfrak{M}}$ is the disjoint union over all birational contractions $\phi: X \dashrightarrow X'$ of the $\CI^{\circ}_{\phi}$.
\end{itemize}
To see this, first recall that for a pseudo-effective divisor $L$ the $\sigma$-decomposition of $L$ and the volume are preserved by $\phi_{*}$.  We know that each $\alpha \in \Mov_{1}(X)_{\mathfrak{M}}$ has the form $\langle L^{n-1} \rangle$ for a unique big and movable divisor $L$.   If $\phi: X \dashrightarrow X'$ denotes the birational canonical model obtained by running the $L$-MMP, and $A$ denotes the corresponding ample divisor on $X'$, then $\phi_{*}\alpha = A^{n-1}$ and $\alpha = \langle \phi^{*}A^{n-1} \rangle$.  The various claims now can be deduced from the properties of divisors and the MMP for Mori Dream Spaces as in \cite[1.11 Proposition]{hk00}.

Since the volume of divisors behaves compatibly with strict transforms of pseudo-effective divisors, the description of $\phi_{*}$ above shows that $\mathfrak{M}$ also behaves compatibly with strict transforms of movable curves under an SQM.  However, the volume function can change: we may well have $\widehat{\vol}(\phi_{*}\alpha) \neq \widehat{\vol}(\alpha)$.  The reason is that the pseudo-effective cone of curves is also changing as we vary $\phi$.  In particular, the set
\begin{equation*}
C_{\alpha,\phi} := \{ \phi_{*}\alpha - \gamma | \gamma \in \Eff_{1}(X') \}
\end{equation*}
will look different as we vary $\phi$.  Since $\widehat{\vol}$ is the same as the maximum value of $\mathfrak{M}(\beta)$ for $\beta \in C_{\alpha,\phi}$, the volume and Zariski decomposition for a given model will depend on the exact shape of $C_{\alpha,\phi}$.
\end{exmple}

\begin{rmk}
Theorem \ref{thm solution to xiao} and Theorem \ref{thm big to nef} also hold for smooth varieties over any algebraically closed field and compact hyperk\"ahler manifolds as explained in Section \ref{section preliminaries}.
\end{rmk}

\section{Comparison with mobility}
\label{section mobility}

In this section we give the proof of the main result, comparing the volume function for curves with its mobility function.
Recall from the introduction that we are trying to show (rearranged in a slightly different order):

\begin{thrm} \label{volversusmob2}
Let $X$ be a smooth projective variety of dimension $n$ and let $\alpha \in \Eff_{1}(X)$ be a pseudo-effective curve class.  Then the following results hold:
\begin{enumerate}
\item $\widehat{\vol}(\alpha) \leq \mob(\alpha) \leq n! \widehat{\vol}(\alpha)$.
\item Assume Conjecture \ref{conj mobandciconj}.  Then $\mob(\alpha) = \widehat{\vol}(\alpha)$.
\item $\widehat{\vol}(\alpha) = \wmob(\alpha)$.
\end{enumerate}
\end{thrm}

The upper bound in the first part improves the related result  \cite[Theorem 3.2]{xiao15}. Before giving the proof, we repeat the following estimate of $\widehat{\vol}$ in \cite{lehmannxiao2015a}.

\begin{prop} \label{multiplicityestimate}
Let $X$ be a smooth projective variety of dimension $n$.  Choose positive integers $\{k_{i} \}_{i=1}^{r}$.  Suppose that $\alpha \in \Mov_{1}(X)$ is represented by a family of irreducible curves such that for any collection of general points $x_{1},x_{2},\ldots,x_{r},y$ of $X$, there is a curve in our family which contains $y$ and contains each $x_{i}$ with multiplicity $\geq k_{i}$.  Then
\begin{equation*}
\widehat{\vol}(\alpha)^{{n-1}/{n}} \geq \mathfrak{M}(\alpha)^{{n-1}/{n}} \geq \frac{\sum_{i} k_{i}}{r^{1/n}}.
\end{equation*}
\end{prop}

This is just a rephrasing of well-known results in birational geometry; see for example \cite[V.2.9 Proposition]{k96}.

\begin{proof}
By continuity and rescaling invariance, it suffices to show that if $L$ is a big and movable Cartier divisor class then
\begin{equation*}
\left( \sum_{i=1}^{r} k_{i} \right) \frac{\vol(L)^{1/n}}{r^{1/n}} \leq L \cdot C.
\end{equation*}
A standard argument (see for example \cite[Example 8.22]{lehmann14}) shows that for any $\epsilon > 0$ and any very general points $\{ x_{i} \}_{i=1}^{r}$ of $X$ there is a positive integer $m$ and a Cartier divisor $M$ numerically equivalent to $mL$ and such that $\mult_{x_{i}}M \geq mr^{-1/n}\vol(L)^{1/n} - \epsilon$ for every $i$.  By the assumption on the family of curves we may find an irreducible curve $C$ with multiplicity $\geq k_{i}$ at each $x_{i}$ that is not contained $M$.  Then
\begin{equation*}
m(L \cdot C) \geq \sum_{i=1}^{r} k_{i} \mult_{x_{i}}M \geq \left(\sum_{i = 1}^{r} k_{i} \right)  \left(\frac{ m\vol(L)^{1/n}}{r^{1/n}} - \epsilon \right).
\end{equation*}
Divide by $m$ and let $\epsilon$ go to $0$ to conclude.
\end{proof}

\begin{exmple}
The most important special case is when $\alpha$ is the class of a family of irreducible curves such that for any two general points of $X$ there is a curve in our family containing them.  Proposition \ref{multiplicityestimate} then shows that $\widehat{\vol}(\alpha) \geq 1$ and $\mathfrak{M}(\alpha) \geq 1$.
\end{exmple}

We also need to give a formal definition of the mobility count.  Its properties are studied in more depth in \cite{lehmann14}.

\begin{defn} \label{mcdefn}
Let $X$ be an integral projective variety and let $W$ be a reduced variety.  Suppose that $U\subset W \times X$ is a subscheme and let $p: U \to W$ and $s: U \to X$ denote the projection maps.  The mobility count $\mc(p)$ of the morphism $p$ is the maximum non-negative integer $b$ such that the map
\begin{equation*}
U \times_{W} U \times_{W} \ldots \times_{W} U \xrightarrow{s \times s \times \ldots \times s} X \times X \times \ldots \times X
\end{equation*}
is dominant, where we have $b$ terms in the product on each side.  (If the map is dominant for every positive integer $b$, we set $\mc(p) = \infty$.)

For $\alpha \in N_{1}(X)_{\mathbb{Z}}$, the mobility count of $\alpha$, denoted $\mc(\alpha)$, is defined to be the largest mobility count of any family of effective curves representing $\alpha$.  We define the rational mobility count $\rmc(\alpha)$ in the analogous way by restricting our attention to rational families.
\end{defn}

The mobility is then defined as
\begin{equation*}
\mob(\alpha) = \limsup_{m \to \infty} \frac{\mc(m\alpha)}{m^{n/n-1}/n!}.
\end{equation*}

\begin{proof}[Proof of Theorem \ref{volversusmob2}:]
(1) We compare $\mob$ and $\widehat{\vol}$. We first prove the upper bound.  By continuity and homogeneity it suffices to prove the upper bound for a class $\alpha$ in the natural sublattice of integral classes $N_{1}(X)_{\mathbb{Z}}$.  Suppose that $p: U \to W$ is a family of curves representing $m\alpha$ of maximal mobility count for a positive integer $m$.  Suppose that a general member of $p$ decomposes into irreducible components $\{ C_{i} \}$; arguing as in \cite[Corollary 4.10]{lehmann14}, we must have $\mc(p) = \sum_{i} \mc(U_{i})$, where $U_{i}$ represents the closure of the family of deformations of $C_{i}$.  We also let $\beta_{i}$ denote the numerical class of $C_{i}$.

Suppose that $\mc(U_{i}) > 1$.  Then we may apply Proposition \ref{multiplicityestimate} with all $k_{i} = 1$ and $r = \mc(U_{i})-1$ to deduce that
\begin{equation*}
\widehat{\vol}(\beta_{i}) \geq \mc(U_{i}) - 1.
\end{equation*}
If $\mc(U_{i}) \leq 1$ then Proposition \ref{multiplicityestimate} does not apply but at least we still know that $\widehat{\vol}(\beta_i) \geq 0 \geq \mc(U_{i})-1$.  Fix an ample Cartier divisor $A$, and note that the number of components $C_{i}$ is at most $mA \cdot \alpha$.  All told, we have
\begin{align*}
\widehat{\vol}(m\alpha) & \geq \sum_{i} \widehat{\vol}(\beta_{i}) \\
& \geq \sum_{i} ( \mc(U_{i}) - 1) \\
& \geq \mc(m\alpha) - mA \cdot \alpha.
\end{align*}
Thus,
\begin{align*}
\widehat{\vol}(\alpha) & = \limsup_{m \to \infty} \frac{\widehat{\vol}(m\alpha)}{m^{n/n-1}} \\
& \geq \limsup_{m \to \infty} \frac{\mc(m\alpha) - mA \cdot \alpha}{m^{n/n-1}} = \frac{\mob(\alpha)}{n!}.
\end{align*}

The lower bound relies on the Zariski decomposition of curves in Theorem \ref{thm curve decomposition}.
By \cite[Theorem 6.11 and Example 6.2]{lehmann14} we have
 $$B^{n} \leq \mob(B^{n-1})$$
for any nef divisor $B$. With Theorem \ref{thm curve decomposition}, this implies
 $$
 \widehat{\vol}(B^{n-1}) \leq \mob(B^{n-1}).
 $$
In general, for a big curve class $\alpha$ we have
\begin{align*}
\mob(\alpha) & \geq \sup_{B \textrm{ nef, }
\alpha \succeq B^{n-1}} \mob(B^{n-1})  \\
& \geq \sup_{B \textrm{ nef, }
\alpha \succeq B^{n-1}} B^n \\
& = \widehat{\vol}(\alpha).
\end{align*}
where the last equality again follows from Theorem \ref{thm curve decomposition}. This finishes the proof.

(2) To prove the second part of Theorem \ref{volversusmob2}, we need a result of \cite{fl14}:

\begin{lem}[see \cite{fl14} Corollary 6.16]
\label{lemma fl pull back mob}
Let $X$ be a smooth projective variety of dimension $n$ and let $\alpha$ be a big curve class.  Then there is a big movable curve class $\beta$ satisfying $\beta \preceq \alpha$ such that
\begin{equation*}
\mob(\alpha) = \mob(\beta) = \mob(\phi^{*}\beta)
\end{equation*}
for any birational map $\phi: Y \to X$ from a smooth variety $Y$.
\end{lem}

We now prove the statement via a sequence of claims.

\begin{claim}
Assume Conjecture \ref{conj mobandciconj}.  If $\beta$ is a movable curve class with $\mathfrak{M}(\beta)>0$, then for any $\epsilon > 0$ there is a birational map $\phi_{\epsilon}: Y_{\epsilon} \to X$ such that
\begin{equation*}
\mathfrak{M}(\beta) - \epsilon \leq \mob(\phi_{\epsilon}^{*}\beta) \leq \mathfrak{M}(\beta) + \epsilon.
\end{equation*}
\end{claim}

By Theorem \ref{thm mov-intersection after}, we may suppose that there is a big divisor $L$ such that $\beta = \langle L^{n-1} \rangle$.
Without loss of generality we may assume that $L$ is effective.  Fix an ample effective divisor $G$ as in\cite[Proposition 6.24]{fl14}; the proposition shows that for any sufficiently small $\epsilon$ there is a birational morphism $\phi_{\epsilon}: Y_{\epsilon} \to X$ and a big and nef divisor $A_{\epsilon}$ on $Y_{\epsilon}$ satisfying
\begin{equation*}
A_{\epsilon} \leq P_{\sigma}(\phi_{\epsilon}^{*}L) \leq A_{\epsilon} + \epsilon \phi_{\epsilon}^{*}G.
\end{equation*}
Note that $\vol(A_{\epsilon}) \leq \vol(L) \leq \vol(A_{\epsilon} + \epsilon \phi_{\epsilon}^{*}G)$.  Furthermore, we have
$$\vol(A_{\epsilon} + \epsilon \phi_{\epsilon}^{*}G) \leq \vol(\phi_{\epsilon *}A_{\epsilon} + \epsilon G) \leq \vol(L + \epsilon G).$$
Applying \cite[Lemma 6.21]{fl14} and the invariance of the positive product under passing to positive parts, we have
\begin{equation*}
A_{\epsilon}^{n-1} \preceq \phi_{\epsilon}^{*}\beta \preceq (A_{\epsilon} + \epsilon \phi_{\epsilon}^{*}G)^{n-1}.
\end{equation*}
Applying Conjecture \ref{conj mobandciconj} (which is only stated for ample divisors but applies to big and nef divisors by continuity of $\mob$), we find
\begin{equation*}
\vol(A_{\epsilon}) = \mob(A_{\epsilon}^{n-1}) \leq \mob(\phi_{\epsilon}^{*}\beta) \leq \mob((A_{\epsilon} + \epsilon \phi_{\epsilon}^*(G))^{n-1}) = \vol(A_{\epsilon} + \epsilon \phi_{\epsilon}^{*}G).
\end{equation*}
As $\epsilon$ shrinks the two outer terms approach $\vol(L) = \mathfrak{M}(\beta)$.

\begin{claim}
Assume Conjecture \ref{conj mobandciconj}.  If a big movable curve class $\beta$ satisfies $\mob(\beta) = \mob(\phi^{*}\beta)$ for every birational $\phi$ then we must have $\beta \in \CI_{1}(X)$.
\end{claim}

When $\mathfrak{M}(\beta) > 0$, by the previous claim we see from taking a limit that $\mob(\beta) = \mathfrak{M}(\beta)$.  By Theorem \ref{volversusmob2}.(1) and Theorem \ref{thm solution to xiao} we get
\begin{equation*}
\widehat{\vol}(\beta) \leq \mathfrak{M}(\beta) \leq \widehat{\vol}(\beta)
\end{equation*}
and Theorem \ref{thm solution to xiao} implies the result.  When $\mathfrak{M}(\beta) = 0$, fix a class $\xi$ in the interior of the movable cone and consider $\beta + \delta \xi$ for $\delta > 0$.  By the previous claim, for any $\epsilon > 0$ we can find a sufficiently small $\delta$ and a birational map $\phi_{\epsilon}: Y_{\epsilon} \to X$ such that $\mob(\phi_{\epsilon}^{*}(\beta + \delta \xi)) < \epsilon$.  We also have $\mob(\phi_{\epsilon}^{*}\beta) \leq \mob(\phi_{\epsilon}^{*}(\beta + \delta \xi))$ since the pullback of the nef curve class $\delta \xi$ is pseudo-effective.  By the assumption on the birational invariance of $\mob(\beta)$, we can take a limit to obtain $\mob(\beta)=0$, a contradiction to the bigness of $\beta$.

To finish the proof, recall that Lemma \ref{lemma fl pull back mob} implies that the mobility of $\alpha$ must coincide with the mobility of a movable class $\beta$ lying below $\alpha$ and satisfying $\mob(\pi^{*}\beta) = \mob(\beta)$ for any birational map $\pi$.  Thus we have shown
\begin{equation*}
\mob(\alpha) = \sup_{B \textrm{ nef, } \alpha \succeq B^{n-1}} \mob(B^{n-1}).
\end{equation*}
By Conjecture \ref{conj mobandciconj} again, we obtain
\begin{equation*}
\mob(\alpha) = \sup_{B \textrm{ nef, } \alpha \succeq B^{n-1}} B^{n}.
\end{equation*}
But the right hand side agrees with $\widehat{\vol}(\alpha)$ by Theorem \ref{thm curve decomposition}. This proves the equality $\mob(\alpha)=\widehat{\vol}(\alpha)$ under the Conjecture \ref{conj mobandciconj}.

(3) We now prove the equality $\widehat{\vol}=\wmob$. The key advantage is that the analogue of Conjecture \ref{conj mobandciconj} is known for the weighted mobility: \cite[Example 8.22]{lehmann14} shows that for any big and nef divisor $B$ we have $\wmob(B^{n-1}) = B^{n}$.

We first prove the inequality $\widehat{\vol} \geq \wmob$.  The argument is essentially identical to the upper bound in Theorem \ref{volversusmob2}.(1): by continuity and homogeneity it suffices to prove it for classes in $N_{1}(X)_{\mathbb{Z}}$.  Choose a positive integer $\mu$ and a family of curves of class $\mu m \alpha$ achieving $\wmc(m\alpha)$.  By splitting up into components and applying Proposition \ref{multiplicityestimate} with equal weight $\mu$ at every point we see that for any component $U_{i}$ with class $\beta_{i}$ we have
\begin{equation*}
\widehat{\vol}(\beta_{i}) \geq \mu^{n/n-1} (\wmc(U_{i}) - 1)
\end{equation*}
Arguing as in Theorem \ref{volversusmob2}.(1), we see that for any fixed ample Cartier divisor $A$ we have
\begin{equation*}
\widehat{\vol}(m\mu \alpha) \geq \mu^{n/n-1} (\wmc(m\alpha) - mA \cdot \alpha).
\end{equation*}
Rescaling by $\mu$ and taking a limit proves the statement.

We next prove the inequality $\widehat{\vol} \leq \wmob$.  Again, the argument is identical to the lower bound in Theorem \ref{volversusmob2}.(1).  It is clear that the weighted mobility can only increase upon adding an effective class.  Using continuity and homogeneity, the same is true for any pseudo-effective class.  Thus we have
\begin{align*}
\wmob(\alpha) & \geq \sup_{B \textrm{ nef, }
\alpha \succeq B^{n-1}} \wmob(B^{n-1})  \\
& = \sup_{B \textrm{ nef, }
\alpha \succeq B^{n-1}} B^n \\
& = \widehat{\vol}(\alpha).
\end{align*}
where the second equality follows from \cite[Example 8.22]{lehmann14}. This finishes the proof of the equality $\widehat{\vol}=\wmob$.
\end{proof}

\begin{rmk}
We expect Theorem \ref{volversusmob2} to also hold over any algebraically closed field, but we have not thoroughly checked the results on asymptotic multiplier ideals used in the proof of \cite[Proposition 6.24]{fl14}.
\end{rmk}

Theorem \ref{volversusmob2} yields two interesting consequences:
\begin{itemize}
\item The theorem indicates (loosely speaking) that if the mobility count of complete intersection classes is optimized by complete intersection curves, then the mobility count of \emph{any} curve class is optimized by complete intersection curves lying below the class.

This result is very surprising: it indicates that the ``positivity'' of a curve class is coming from ample divisors in a strong sense.  For example, suppose that $X$ and $X'$ are isomorphic in codimension $1$.  If we take a complete intersection class $\alpha$ on $X$, we expect that complete intersections of divisors maximize the mobility count.  However, the strict transform of these curves on $X'$ should not maximize mobility count.  Instead, if we deform these curves so that they break off a piece contained in the exceptional locus, the part left over deforms more than the original.

\item The theorem suggests that the Zariski decomposition constructed in \cite{fl14} for curves is not optimal: instead of defining a positive part in the movable cone, if Conjecture \ref{conj mobandciconj} is true we should instead define a positive part in the complete intersection cone. It would be interesting to see an analogous improvement for higher dimension cycles.
\end{itemize}

\section{$\mathfrak{M}$ and asymptotic point counts}

Finally, we show that $\mathfrak{M}$ can be given an enumerative interpretation.

\begin{defn}
Let $p: U \to W$ be a family of curves on $X$ with morphism $s: U \to X$.  We say that $U$ is strictly movable if:
\begin{enumerate}
\item For each component $U_{i}$ of $U$, the morphism $s|_{U_{i}}$ is dominant.
\item For each component $U_{i}$ of $U$, the morphism $p|_{U_{i}}$ has generically irreducible fibers.
\end{enumerate}
\end{defn}

We then define $\mob_{mov}$ and $\wmob_{mov}$ exactly analogously to $\mob$ and $\wmob$, except that we only allow contributions of strictly movable families of curves.  Note that $\mob_{mov}$ and $\wmob_{mov}$ vanish outside of $\Mov_{1}(X)$ since these classes are not represented by a sum of irreducible curves which deform to  dominate $X$.  Arguing just as in \cite[Section 5]{lehmann14}, one sees that $\mob_{mov}$ and $\wmob_{mov}$ are homogeneous of weight $n/n-1$, and are continuous in the interior of $\Mov_{1}(X)$.

\begin{lem} \label{birneglem}
Let $\phi: Y \to X$ be a birational morphism of smooth projective varieties.  Let $p: U \to W$ be a family of irreducible curves admitting a dominant map $s: U \to X$.  Let $U_{Y}$ be the family of curves defined by strict transforms.  Letting $\alpha, \alpha_{Y}$ denote respectively the classes of the families on $X,Y$, we have that $\phi^{*}\alpha - \alpha_{Y}$ is the class of an effective $\mathbb{R}$-curve.
\end{lem}

\begin{proof}
Since $\alpha_{Y}$ is the class of a family of irreducible curves which dominates $Y$, it has non-negative intersection against every effective divisor.  Arguing as in the negativity of contraction lemma, we can find a basis $\{ e_{i} \}$ of $\ker(\phi_{*}:N_{1}(Y) \to N_{1}(X))$ consisting of effective curves and a basis $\{ f_{j} \}$ of $\ker(\phi_{*}: N^{1}(Y) \to N^{1}(X))$ consisting of effective divisors such that the intersection matrix is negative definite and the only negative entries are on the diagonal.  Just as in \cite[Lemma 4.1]{bck12}, this shows that
\begin{equation*}
\alpha_{Y} = \phi^{*}\phi_{*}\alpha_{Y} - \beta = \phi^{*}\alpha - \beta
\end{equation*}
for some effective curve class $\beta$ supported on the exceptional divisors.
\end{proof}

\begin{thrm} \label{movversion}
Let $X$ be a smooth projective variety of dimension $n$ and let $\alpha \in \Mov_{1}(X)^{\circ}$.  Then:
\begin{enumerate}
\item $\mathfrak{M}(\alpha) = \wmob_{mov}(\alpha)$.
\item Assume Conjecture \ref{conj mobandciconj}.  Then $\mathfrak{M}(\alpha) = \mob_{mov}(\alpha)$.
\end{enumerate}
\end{thrm}

\begin{proof}
(1) Suppose that $\phi: Y \to X$ is a birational model of $X$ and that $A$ is an ample Cartier divisor on $X$.  By pushing-forward complete intersection families, we see that $\wmob_{mov}(\phi_{*}A^{n-1}) \geq A^{n}$.  By continuity we obtain the inequality $\mathfrak{M}(\alpha) \leq \wmob_{mov}(\alpha)$ for any $\alpha \in \Mov_{1}(X)^{\circ}$.

To see the reverse inequality, by continuity and homogeneity it suffices to consider the case when $\alpha \in \Mov_{1}(X)_{\mathbb{Z}}^{\circ}$.  Choose a positive integer $\mu$ and a strictly movable family of curves $U$ of class $\mu m \alpha$ achieving $\wmc_{mov}(m\alpha)$.  Let $\phi: Y \to X$ be a birational model and let $U_{Y}$ denote the strict transform class on $Y$ with numerical class $\alpha'$.  By arguing as in the proof of Theorem \ref{volversusmob2}, we find that
\begin{equation*}
\mathfrak{M}(\alpha') \geq \mu^{n/n-1}(\wmc_{mov}(m\alpha) - mA \cdot \alpha).
\end{equation*}
Furthermore by Lemma \ref{birneglem} we have $\widehat{\vol}(m\mu \phi^{*}\alpha) \geq \widehat{\vol}(\alpha')$.  Dividing by $m^{n/n-1}$ and taking a limit as $m$ increases, we see that $\mathfrak{M}(\alpha) \geq \wmob_{mov}(\alpha)$.

(2) The proof of $\mathfrak{M}(\alpha) \leq \mob_{mov}(\alpha)$ is the same as in (1).
Conversely, suppose that $U$ is a strictly movable family of curves achieving $\mc_{mov}(m\alpha)$.  Let $\phi: Y \to X$ be a birational morphism of smooth varieties; by combining Lemma \ref{birneglem} with \cite[Section 4]{fl14}, we see that $\mc_{mov}(m\alpha) \leq \mc_{\mathcal{K}}(m\phi^{*}\alpha)$, where $\mathcal{K}$ is a cone chosen as in \cite[Definition 4.8]{fl14} and includes a fixed effective basis of the kernel of $\phi_{*}: N_{1}(Y) \to N_{1}(X)$ chosen as in Lemma \ref{birneglem}.  Taking limits, we see that $\mob_{mov}(\alpha) \leq \mob(\phi^{*}\alpha)$ for any birational map $\phi$.

Choose a sequence of birational maps $\phi_{i}: Y_{i} \to X$ as in the proof of Proposition \ref{prop hatvol birational} so that $\widehat{\vol}(\phi_{i}^{*}\alpha)$ limits to $\mathfrak{M}(\alpha)$.  By taking a limit over $i$ and applying Theorem \ref{volversusmob2}.(2) we finish the proof.
\end{proof}

\bibliography{poscurves}
\bibliographystyle{amsalpha}

\noindent
\textsc{Brian Lehmann}\\
\textsc{Department of Mathematics, Boston College,
Chestnut Hill, MA 02467, USA}\\
\verb"Email: lehmannb@bc.edu"\\

\noindent
\textsc{Jian Xiao}\\
\textsc{Institute of Mathematics, Fudan University, 200433 Shanghai, China}\\

\noindent
\textsc{Current address:}\\
\textsc{Institut Fourier, Universit\'{e} Joseph Fourier, 38402 Saint-Martin d'H\`{e}res, France}\\
\verb"Email: jian.xiao@ujf-grenoble.fr"
\end{document}